\documentclass{article}

\usepackage{mathptmx}
\usepackage[T1]{fontenc}
\usepackage{amsmath, amsthm, amsfonts}
\usepackage{amssymb}
\usepackage{graphicx}
\usepackage[left=3cm,right=4cm, top=4cm, bottom=3.5 cm]{geometry}
\newtheorem{thm}{Theorem}[section]
\newtheorem{lem}[thm]{Lemma}
\newtheorem{sublem}[thm]{Sublemma}
\newtheorem{prop}[thm]{Proposition}
\newtheorem{prob}[thm]{Conjecture}
\newtheorem{rem}[thm]{Remark}
\theoremstyle{definition}
\newtheorem{defn}[thm]{Definition}
\theoremstyle{remark}

\newcommand\realdensity[1]{real $(#1)$-density }
\def\realstardensity{real $*$-density}
\def\realstardense{real $*$-dense}
\newcommand\intervalN[2]{[#1\cdots #2)}
\def\R{\mathbb{R}}
\def\Z{\mathbb{Z}}
\def\N{\mathbb{N}}
\def\ab{\hspace*{2pt}}
\def\F{\mathbb{F}}
\def\n{0}
\def\e{1}
\def\m{\theta}
\def\ColN{C_\N}
\def\Col_R{C_\R}
\title{Almost all orbits of an analogue of the Collatz map on the reals attain bounded values}
\author{Manuel Inselmann
}
\usepackage[headsepline]{scrlayer-scrpage}
\usepackage{setspace}

\begin{document}
	\maketitle
	\begin{abstract}
		\noindent Motivated by a balanced ternary representation of the Collatz map we define the map $C_\mathbb{R}$ on the positive real numbers by setting  $C_\mathbb{R}(x)=\frac{1}{2}x$ if $[x]$ is even and $C_\mathbb{R}(x)=\frac{3}{2}x$ if $[x]$ is odd, where $[x]$ is defined by $[x]\in\mathbb{Z}$ and $x-[x]\in(-\frac{1}{2},\frac{1}{2}]$. We show that there exists a constant $K>0$ such that the set of $x$ fulfilling  $\liminf_{n\in\N}C_\mathbb{R}^n(x)\leq K$ is Lebesgue-co-null. We also show that for any $\epsilon>0$ the set of $x$ for which $  (\frac{3^{\frac{1}{2}}}{2})^kx^{1-\epsilon}\leq C_\mathbb{R}^k(x)\leq (\frac{3^{\frac{1}{2}}}{2})^kx^{1+\epsilon}$ for all $0\leq k\leq \frac{1}{1-\frac{\log_23}{2}}\log_2x$ is large for a suitable notion of largeness.
	\end{abstract}

	\section{Introduction}
	The Collatz map $\text{Col}:\N+1\rightarrow\N+1$ takes an even number $n$ to $\frac{n}{2}$ and an odd number to $3n+1$. For $n\in\N+1$ define $\text{Col}_{\min}(n)=\min_{k\in\N}\text{Col}^k(n)$. The \textit{Collatz conjecture} states that $Col_{min}(n)=1$ for every $n\in\N+1$. Note that whenever $n\in\N$ is odd, then $\text{Col}(n)$ is even. Thus one can define a version of $\ColN:\N+1\rightarrow\N+1$ which takes $n$ to $\frac{1}{2}n$ if $n$ is even and to $\frac{3n+1}{2}$, if $n$ is odd. An heuristic approach suggests that when we iterate $\ColN$ for $k$ times, then roughly $\frac{k}{2}$ of the set $\{\ColN^i\mid 0\leq i<k\}$  will be odd and the other half will be even, thus expecting $\ColN^k(n)$ to be close to $(\frac{3^{\frac{1}{2}}}{2})^kn$. See \cite{lagarias} for an overview of this problem. In \cite{Tao} Terrence Tao showed that for any function $f:\N\rightarrow \R$ such that $\lim_{n\rightarrow\infty}f(n)=\infty$, the set of $n$ such that $\text{Col}_{\min}(n)<f(n)$ is of logarithmic density $1$.
	
	In this paper we look at an analogue of $\text{Col}$ defined on the positive real numbers, which is more amenable to an analysis, but at the same time shows a similar behavior to the original Collatz map. If $x\in(0,\infty)$, let $[x]$ be the closest natural number to $x$, i.e., $x-[x]\in(-\frac{1}{2},\frac{1}{2}]$. Now, define $\Col_R(x):=\frac{x}{2}$ if $[x]$ is even and $\Col_R(x):=\frac{3x}{2}$ if $[x]$ is odd. The motivation for this map is coming from the balanced ternary representation of the real numbers: Each real number $x$ can be written as $x=\sum_{k\in\Z}a_k3^k$ where $a_k\in\{\e,\n,\m\}$ and $\#\{k\in\N\mid a_k\neq\n\}<\infty$ (we follow the convention to denote $-1$ by $\m$). This representation is unique, unless it ends in a constants sequence of $\e$s or a constant sequence of $\m$s, in which case there a two representations (in the following we will tacitly assume that we always choose the representation that ends in a constant sequence of $\e$s when there are two options). This representation allows for a neat formulation of the Collatz map as a map on the set of words in the alphabet $\Sigma=\{\e,\n,\m\}$ (which to the author's knowledge has not been published): Let a word of the form $a_na_{n-1}\cdots a_0$ represent the integer $\sum_{k=0}^na_k3^k$ (with $a_k\in\Sigma$). If $w$ is a word we recursively define $w^{*n}$ for $n\in\N$ to be the empty word, when $n=0$ and $w^{*(n+1)}=w^{*n}w$. Denote $a(0)^{*n}b$ by $B^n_{a,b}$ for $a,b\in\{\e,\m\}$ and $n\in\N$ (in this paper we include $0\in\N$). Note that a word $a_na_{n-1}\cdots a_0$ represents an even number if and only if $\sum_{k=0}^{n}a_k$ is even. Thus the representation of each positive even number $m$ can be uniquely written as $B_1\n^{*n_1}B_2\n^{*n_2}...B_k\n^{*n_k}$, where each $B_i$ is a block of the form $B^n_{a,b}$ and the $n_i$ are natural numbers (we assume that the first letter in the representation is not equal to $\n$). Now, the representation of $\frac{m}{2}$ can be derived as follows: Each block of the form $\n^{*n_i}$ is left unchanged. If a block is of the form $\e(\n)^{*j}\m$, then replace it by $\n(\e)^{*(j+1)}$, if it is of the form $\e(\n)^{*j}\e$, then replace it by $\e(\m)^{*{j+1}}$, if it is of the form $\m(\n)^{*j}\e$, then replace it by $\n(\m)^{*(j+1)}$, and if it is of the form $\m(0)^{*j}\m$, then replace it by $\m(\e)^{*(j+1)}$. That the resulting word represents $\frac{m}{2}$ follows from the fact that $\frac{3^j-1}{2}=\sum_{k=0}^{j-1}3^k$ and $\frac{3^j+1}{2}=3^j-\sum_{k=0}^{j-1}3^k$ for all $j\in\N$. If $w$ represents a odd number $m$ then $w\e$ represents $3m+1$. Thus the Collatz map $\ColN$ on the positive natural numbers corresponds to the map $C$ defined on the words $w$ with letters from $\Sigma$ beginning with $\e$  and applying the above procedure, if $w$ has an even number of letters distinct from $\n$ and applying the procedure to $w\e$ otherwise. Finally, if now the first letter is $\n$, then delete it.
	
	 As an example take $\e\n\m\n\e\n\n\e\m$. Since the parity of the number of non-zero letters is $1$, we add a $\e$ to the right and obtain  $(\e\n\m)\n(\e\n\n\e)(\m\e)$ as the block representation. Applying the algorithm we obtain $(\n\e\e)\n(\e\m\m\m)(\n\m)$ and deleting the leading zero we get to the final result  $\e\e\n\e\m\m\m\n\m$. In this fashion the sequence $3,5,8,4,2,1$ corresponds to
	$\e\n,\e\m\m, \e\n\m,\e\e, \e\m, \e$.
	
	Instead of looking at $\Sigma=\{\e,\n,\m\}$ one may look at $\Sigma^\prime=\{\e,\n\}$ and define a map on the set words with letters in $\Sigma^\prime$ as follows: Denote $\e(0)^{*n}\e$ by $B^n$ for $n\in\N$. Call a word \textit{even}, if it contains an even number of $\e$s. Every even word can be written as ${\n^{*n_0}}B_1{\n^{*n_1}}B_{\n^{*n_2}}...B_k{\n^{*n_k}}$, where each $B_i$ is a block of the form $B^n$ and the $n_i$ are natural numbers. Now define a map $T$ as follows: 
	Each block of the form $\n^{*n_i}$ is left unchanged and if a block is of the form $\e(\n)^{*j}\e$, replace it by $\n(\e)^{*(j+1)}$. If a word is not even, put a $\e$ to its right end and apply the above algorithm. This map can be represented in as map on $\F_2[x]$: Under the correspondence that sends $a_na_{n-1}\cdots a_0$ to the polynomial $\sum_{k=0}^na_kx^k$ and noting that $\frac{1+x^n}{1+x}=\sum_{k=0}^{n-1}x^k$ we have that $T$ corresponds to the map $S:\F_2[x]\rightarrow\F_2[x]$ defined by 
	$$S(f)=\begin{cases}
		\frac{f}{x+1}&\text{\ab if\ab} f(1)=0,\\
		\frac{xf+1}{x+1}&\text{\ab if\ab} f(1)=1.
	\end{cases}$$\\
	Now $S$ is conjugate via the isomorphism $f(x)\mapsto f(x+1)$ to the map $S_0:\F_2[x]\rightarrow\F_2[x]$ defined by 
	$$S_0(f)=\begin{cases}
		\frac{f}{x}&\text{\ab if\ab} f(0)=0,\\
		\frac{(x+1)f+1}{x}&\text{\ab if\ab} f(0)=1.
	\end{cases}$$\\
And $S_0$ is a accelerated version of the map
$S_1:\F_2[x]\rightarrow\F_2[x]$ defined by 
 	$$S_1(f)=\begin{cases}
 	\frac{f}{x}&\text{\ab if\ab} f(0)=0,\\
 	(x+1)f+1&\text{\ab if\ab} f(0)=1.
 \end{cases}$$\\
$S_1$ was first studied in \cite{Hicks} as a polynomial analogue of the Collatz map and further investigated in (among others) \cite{alon}. The above derivation of it gives even more reason to consider it an analogue of the Collatz map. The dynamics of $S$ is studied in \cite{inselmann}.

	One aspect that makes the Collatz map  difficult to analyze is the addition of $1$ in the odd case. The balanced ternary approach suggests a map where this difficulty vanishes by passing to words that are infinite to the right, which correspond to the real numbers: If $x\in\R$ has the balanced ternary representation $x=\sum_{k\in\Z}a_k3^k$ where $a_k\in\{\e,\n,\m\}$ and $N$ is the maximal non-$\n$ coefficient, then we may write $x$ as $a_N\cdots a_0\cdot a_{-1}a_{-2}\cdots$. Then the block representation is always possible except for the countable set, where the representation ends in $\n^{*\infty}$ and the parity is odd. In this case we send a block of the form $\e(\n)^{*\infty}$ to $0\e^{*\infty}$ and a block of the form $\m(\n)^{*\infty}$ to $\m\e^{*\infty}$. Then what $\Col_R$ does is applying the above procedure and moving $\cdot$ one step to the right, exactly when it lies inside a block of the form $B^n_{a,b}$ for $a,b\in\{\e,\m\}$ (or of the form $\e(\n)^{*\infty}$ or $\m(\n)^{*\infty}$).

	As in the original problem we may ask what happens if we iterate $\Col_R$.  We restrict ourselves to positive real numbers, since -- unlike $\ColN$ -- we have $\Col_R(-r)=-\Col_R(r)$ (except for a countable set of sequences ending in $\e^{\infty}$, which does not matter for the long term behavior since each orbit of $\Col_R$ contains at most on of these exceptional points). Furthermore, $\Col_R([0,\frac{1}{2^n}))= [0,\frac{1}{2^{n+1}})$ for $n\in\N+1$. Thus we can restrict ourselves to the interval $(\frac{1}{2},\infty)$.

	Note that $\Col_R((\frac{3}{4},\frac{3}{2}])=(\frac{9}{8},\frac{9}{4}]$ and $\Col_R((\frac{3}{2},\frac{9}{4}])=(\frac{3}{4},\frac{9}{8}]$. Thus $\Col_R((\frac{3}{4},\frac{9}{4}])=(\frac{3}{4},\frac{9}{4}]$ and $\Col_R$ is bijective when restricted to $(\frac{3}{4},\frac{9}{4}]$. Thus $\Col_R$ restricted to $(\frac{3}{4},\frac{9}{4}]$  resembles the restriction of $\ColN$ to $\{1,2\}$. 
	
	If $r\in (\frac{1}{2},\frac{3}{4})$ then $\Col_R(r)\in (\frac{3}{4},\frac{9}{8}]$. The obvious question now is if this interval is also always reached when starting with greater values:
	\begin{prob}\label{totaltree}
		Suppose that $r>\frac{9}{4}$. Then there is a $n>0$ such that $\Col_R^n(r)\in (\frac{3}{4},\frac{9}{4}]$.
	\end{prob}	
	This seems a much more tractable question than the analogous one for $\ColN$. In this paper we give a partial result. Recall that for Lebesgue-measurable subsets $A\subseteq B$ of $\R$, $A$ is \textit{Lebesgue-co-null in} $B$ if $\lambda(B\setminus A)=0$.
	\begin{thm}\label{K}
		There exists $K>0$ such that $\{x\in{(\frac{3}{4},\infty)}\mid \liminf_{n\in\N} \Col_R^n(x)\leq K\}$ is a Lebesgue-co-null subset of $(\frac{3}{4},\infty)$.
	\end{thm}
	For the notion of meagerness an even stronger result holds:
	\begin{thm}\label{comeager subset}
		The set $\{x\in (\frac{3}{4},\infty)\mid \liminf_{n\rightarrow\infty}\Col_R(x)=\frac{3}{4}\}$ is a comeager subset of $(\frac{3}{4},\infty)$. 
	\end{thm}
Theorem \ref{K} leaves the question open what $K>0$ can be chosen. The best choice would be $K=\frac{3}{4}$. Thus we formulate the following conjecture. 
	\begin{prob}\label{verylikely}
		The set $\{x\in (\frac{3}{4},\infty)\mid \liminf_{n\in\N} C^n(x)=\frac{3}{4}\}$ is a Lebesgue-co-null subset of $(\frac{3}{4},\infty)$.
	\end{prob}
	Although we do not know a proof of Conjecture \ref{verylikely},  we show that its truth is equivalent to a problem that is verifiable by a finite (although very large) number of calculations: It is sufficient to check that for a subset of $(\frac{3}{4},R)$ of measure $>\alpha$ we have that $\Col_R^M(x)\leq \frac{9}{8}$ for some $M\in\N$, some large $R$, and some appropriate $0<\alpha<1$, which is possible to do since $\Col_R$ is 'sufficiently continuous' in the sense that $\inf_{n\in\N}\Col_R^M(x)\leq \frac{9}{8}$ is true if it is true for all $y$ in an open neighborhood of $x$ (for most $x$ at least).  Instead of checking sufficiently many numbers one can independently test a finite number of $x\in(\frac{3}{4},R)$. A test with concrete values for $\alpha, R,C$  was performed and every single result was in the desired set. Assuming that the conjecture is false, i.e., the desired set has probability $\leq\alpha$ this outcome would have a probability of less than $10^{-100}$ giving very good reason that Conjecture \ref{verylikely} is true.

Our second type of results concerns approximations of the dynamics of $\Col_R$. To state these results we need the following notion of largeness which will become useful in the analysis:
	\begin{defn}\label{density}
		We say that a subset $A\subseteq\R$ has \textit{real density} $0\leq \lambda\leq 1$ if for every $\epsilon>0$ there exists $R_0>0$ such that $|\frac{\lambda(A\cap(0,R))}{R}-\alpha|<\epsilon$ for all $R\geq R_0$.\\
	We say that a subset $A\subseteq\R$ has \textit{\realdensity{C,D}}, if $C>0$, $0<D<1$, and $\frac{\lambda((0,R)\cap A)}{R}\geq 1-\frac{C}{R^D}$ for all $R\in(0,\infty)$. We say that $A$ is \textit{\realstardense} if it has \realdensity{C,D} for some $C>0$ and $0<D<1$. Note that any \realstardense\ab  set has real density $1$. 
	
	For $K>0$ and $x\in(\frac{3}{4},\infty)$ define its \textit{stopping time}  $\tau_K(x)$ to be the minimal $n\in\N$ such that $\Col_R^n(x)\leq K$ if such an $n$ exists and set $\tau_K(x)=\infty$ otherwise.
	
\end{defn}
Now we can state the other main results:

\begin{thm}\label{maintheoremdynamicintroductoin}
	Suppose that $\epsilon>0$. Then the set $\{x\in(0,\infty)\mid\forall 0\leq k\leq \frac{1}{1-\frac{\log_23}{2}}\log_2x:  (\frac{3^{\frac{1}{2}}}{2})^kx^{1-\epsilon}\leq \Col_R^k(x)\leq (\frac{3^{\frac{1}{2}}}{2})^kx^{1+\epsilon}\}$ is \realstardense.
\end{thm}
\begin{thm}\label{stoppingtime}
	There exists $K>0$ such that the set $\{x\in(\frac{3}{4},\infty)\mid (\frac{1}{1-\frac{\log_23}{2}}-\epsilon)\log_2x\leq \tau_K(x)\leq (\frac{1}{1-\frac{\log_23}{2}}+\epsilon)\log_2x\}$ has real density $1$ for every $\epsilon>0$.
\end{thm}

In the last section we will associate an acyclic directed graph $G_x$ to each $x\in \R$ such that Conjecture \ref{verylikely} is true if and only if the set of $x\in (\frac{3}{4},\infty)$ such that $G_x$ has exactly one connected component is Lebesgue-co-null in  $(\frac{3}{4},\infty)$. 
\section{Almost all reals attain bounded values}
	In the introduction we already noticed that  $\Col_R$ bijectively maps $I_0=(\frac{3}{4},\frac{9}{4}]$ to itself.
	We denote the restriction of $\Col_R$ to $I_0$ by $T$. This map $T$ has some very nice properties, which will be of great help to obtain the main results. We gather them in the following proposition:
	\begin{prop}\label{T}
		$T$ is uniquely ergodic, i.e., there exists a unique probability measure $\mu$ on the Borel subsets of $I_0$ with the property that $\mu(T^{-1}(A))=\mu(A)$ for every Borel set $A\subseteq I_0$.  The corresponding invariant measure $\mu$ is equivalent to the Lebesgue measure restricted to $I_0$. Furthermore, for all $x\in I_0$ the sets $\{T^n(x)\mid n\in\N\}$ and $\{T^{-n}(x)\mid n\in\N\}$ are dense in $I_0$.
	\end{prop}
	\begin{proof}
		The proposition follows from the fact that $T$ is conjugate via a homeomorphism to an irrational rotation and it is well known that irrational rotations are uniquely ergodic and are minimal:
		Via the homeomorphism $\log_3(\cdot):I_0\rightarrow (1-2\log_3(2),2-2\log_3(2)]$,
		$T$	is conjugate to the map: $\tau:(1-2\log_3(2),2-2\log_3(2)]\rightarrow(1-2\log_3(2),2-2\log_3(2)] $ defined by \\$\tau(x)=\begin{cases}
			x+1-log_3({2})& \text{if}\ab x\in (1-2\log_3({2}),1-\log_3({2})] \\
			x-log_3({2})& \text{if}\ab x\in (1-\log_3({2}),2-2\log_3({2})] .
		\end{cases}$\\
		and this in conjugate via the translation given by $x\mapsto x-1+2\log_3(2)$ to:
		\\$\tau^\prime(x)=\begin{cases}
			x+1-log_3({2})& \text{if}\ab x\in (0,\log_3({2})] \\
			x-log_3({2})& \text{if}\ab x\in (\log_3({2}),1] .
		\end{cases}$\\
		Now $\tau^\prime$ is an irrational rotation and the claim follows noting that the push-forward of the Lebesgue measure restricted to $(0,1]$ under the conjugate map has Lebesgue density $f(x)=\frac{1}{x\log(3)}$.
	\end{proof}
	The following notion is simple but will be of essential help later.
	\begin{defn}
		Define	$x\preccurlyeq_3 y$ if there exists $n\in \N$ such that $3^n\cdot x=y$. We will also sometimes write $y\succcurlyeq_3 x$ if  $x\preccurlyeq_3 y$.
	\end{defn}
	We gather some if its properties in a lemma:
	\begin{lem}\label{liminf}
		For every $x\preccurlyeq_3 y$: 
		\begin{enumerate}
			\item  $\Col_R^k(x)\preccurlyeq_3\Col_R^k(y)$ for all $k\in\N$,
			\item 	$\inf_{n\in\N} \Col_R^n(x)\leq\inf_{n\in\N} \Col_R^n(y)$,
			\item 	$\liminf_{n\in\N} \Col_R^n(x)\leq\liminf_{n\in\N} \Col_R^n(y)$.
		\end{enumerate}
		
	\end{lem} 
\begin{proof}
	To see $1.$ it suffices to consider the case $y=3x$ and $k=1$. But by definition $\frac{\Col_R(3x)}{\Col_R(x)}\in\{1,3,9\}$ and the claim follows. $2.$ and $3.$ are immediate consequences from $1.$
\end{proof}
	Now we are already in a position to prove (a slightly stronger version of) Theorem \ref{comeager subset}.
	\begin{thm}\label{comeager}
		The set  $\{x\in(\frac{3}{4},\infty)\mid \liminf_{n\in\N} C^n(x)=\frac{3}{4}\}$ is comeager in $
	(\frac{3}{4},\infty)$ and the set $\{x\in(\frac{3}{4},\frac{9}{4}]\mid \forall k\in\N: \liminf_{n\in\N} C^n(3^kx)=\frac{3}{4}\}$ is comeager in $(\frac{3}{4},\frac{9}{4}]$.
	\end{thm}
	\begin{proof}
		We show that every non-empty open set $V$ of $(\frac{3}{4},\infty)$ contains a non-empty open subset $U\subseteq V$ such that $\liminf_{n\in\N} C^n(x)=\frac{3}{4}$ for all $x\in U$. By making $V$ smaller we can assume that $3^{-i}V\subseteq (\frac{3}{4},\frac{9}{8}]$ for some $i\in\N$. By making $V$ even smaller we can assume that $3^{-i}V=3^{-i-j}(N-\frac{1}{2},N+\frac{1}{2})$ for some $j,N\in\N$. Since $\{T^{-n}(1)\mid n\in\N\}$ is dense in $(\frac{3}{4},\frac{9}{4}]$, we can find $M\in\N$ such that $T^{-M}(1)\in 3^{-i}V$ and $2^M>3^{i+2}N$. If the balanced ternary representation of $N$ is $\sum_{k=0}^la_k3^k$ for some $l\in\N$ and $a_k\in\{\e,\m,\n\}$ with $a_l=\e$, then $T^{-M}(1)\in 3^{-i}V$ and $2^M>N$ implies that the balanced ternary representation of $2^M$ begins with that of $N$, i.e., if $2^M=\sum_{k=0}^pb_k3^k$ for some $p\in\N$ and $a_j\in\{\e,\m,\n\}$ with $a_p=\e$, then $b_{p-k}=a_{l-k}$ for $0\leq k\leq l$. This implies that $3^{l-p}(2^M-\frac{1}{2},2^M+\frac{1}{2}) \subseteq (N-\frac{1}{2},N+\frac{1}{2})$. Thus $3^{l+i-p}(2^M-\frac{1}{2},2^M+\frac{1}{2}) \subseteq V$ but $l+i\leq \log_3(3^i3N)\leq \log_3 \frac{2^M}{3}\leq p$, thus $l+i\leq p$. But by definition for every $x\in(2^M-\frac{1}{2},2^M+\frac{1}{2}) $ we have $\Col_R^M(x)=\frac{x}{2^M}\in(\frac{3}{4},\frac{9}{4}]$. By $1.$ of Lemma \ref{liminf} this holds as well for every $x\in3^{l+i-p}(2^M-\frac{1}{2},2^M+\frac{1}{2})$. Thus we have our desired subset by noticing that $\liminf_{n\in\N} C^n(x)=\frac{3}{4}$ for every $x\in (\frac{3}{4},\frac{9}{4}]$. Thus the union $W$ of all open sets $U$ with the property that $\liminf_{n\in\N} C^n(x)=\frac{3}{4}$  for all $x\in U$ is dense open, thus comeager. Thus for all $n\in\N$ the set $W\cap 3^n(\frac{3}{4},\frac{9}{4}]$ is comeager in $3^n(\frac{3}{4},\frac{9}{4}]$ thus he set $3^{-n}W\cap(\frac{3}{4},\frac{9}{4}]$ is comeager in $(\frac{3}{4},\frac{9}{4}]$ for all $n\in\N$ hence $\bigcap_{n\in\N}(3^{-n}W\cap(\frac{3}{4},\frac{9}{4}])$ is also comeager in $(\frac{3}{4},\frac{9}{4}]$ which proves the second part.
	\end{proof}
	
	Analogously to the original Collatz map we define the \textit{parity sequence} of a real number $x$ to be the sequence $(p(x)_n)_{n\in\N}$, where $$p(x)_n=\begin{cases}
		0& \text{if}\ab [\Col_R^n(x)]\text{\ab is\ab even},\\
		1& \text{if}\ab [\Col_R^n(x)]\text{\ab is\ab odd}.
	\end{cases}$$\\
We have the following lemma.
	\begin{lem}\label{parity}
		$\Col_R^N(x)=\frac{3^{\sum_{n=0}^{N-1}p(x)_n}}{2^N}\cdot x$ for all $N\in\N$ and $x\in\R$.
	\end{lem} 
	\begin{proof}
		We proceed by induction, the case $N=0$ being trivially true. Then	$\Col_R^{N+1}(x)=\frac{3^{p(\Col_R^N(x))_0}}{2}\Col_R^N(x)=\frac{3^{p(x)_N}}{2}\frac{3^{\sum_{n=0}^{N-1}p(x)_n}}{2^N}\cdot x=\frac{3^{\sum_{n=0}^{N}p(x)_n}}{2^N}\cdot x$.  
	\end{proof}
	Now we show that - as in the classical case - the parity sequences of $x$ and $x+2^N\cdot z$ do coincide for the first $n$ values.
	\begin{lem}\label{parity2}
		Suppose that $x\in\R$, $z\in\Z$, $N\in\N$. Then the following hold:
		\begin{enumerate}
			\item $p(x)_n=p(x+2^N\cdot z)_n$ for all $0\leq n< N$.
			\item In particular, $\Col_R^n(x+2^N\cdot z)=\Col_R^n(x)+3^{\sum_{i=0}^{n-1}p(x)_i}2^{N-n}\cdot z$ for all , and $0\leq n\leq N$.
		\end{enumerate}
	\end{lem}
	\begin{proof}
		The second statement follows from the first together with Lemma \ref{parity}.
		We proof the first statement by induction. It is true for all  $x\in\R$, $z\in\Z$, $N\in\N$ and $n=0$. If it is true for $0<n<N$, then we obtain by Lemma \ref{parity}:  $\Col_R^{n}(x+2^N\cdot z)=\frac{3^{\sum_{i=0}^{n-1}p(x+2^N\cdot z)_i}}{2^n}(x+2^{N}\cdot z)$ and again by Lemma \ref{parity} and by induction hypothesis we conclude that $\Col_R^{n}(x+2^N\cdot z)=\frac{3^{\sum_{i=0}^{n-1}p(x+2^N\cdot z)_i}}{2^n}(x+2^{N}\cdot z)=\frac{3^{\sum_{i=0}^{n-1}p(x)_i}}{2^n}(x+2^{N}\cdot z)=\Col_R^n(x)+3^{\sum_{i=0}^{n-1}p(x)_i}2^{N-n}\cdot z$. Since $n<N$, we have $3^{\sum_{i=0}^{n-1}p(x)_i}2^{N-n}\cdot z\in 2\Z$ and thus $p(x)_n=p(x+2^n\cdot z)_n$.
	\end{proof}
In the following let $\intervalN{a}{b}$ denote the set $[a,b)\cap\N$ and let $I^J$ denote the set of all sequences from $J$ to $I$.
	\begin{prop}
		Suppose that $N\in\N$, $r\in\R$, and $(a_n)_{n< N}\in\{0,1\}^{\intervalN{0}{N}}$, then $\lambda(\{x \in [r,r+2^N)\mid (p(x)_n)_{n< N}=(a_n)_{n< N}\})=1$.
	\end{prop}
	\begin{proof}
		Again, we proceed by induction. The case $N=0$ is trivially true. Suppose that the statement is true for $N$ and $r\in \R$ and every $(c_n)_{n< N}\in\{0,1\}^{\intervalN{0}{N}}$. Suppose that $(a_n)_{n< N+1}\in\{0,1\}^{\intervalN{0}{N+1}}$. By hypothesis $\lambda(\{x \in [r,r+2^N)\mid (p(x)_n)_{n< N}=(a_n)_{n< N}\})=1$. The set $\{x \in [r,r+2^N)\mid (p(x)_n)_{n< N+1}=(a_n)_{n< N+1}\}$ is a Borel set and thus measurable. Let $\alpha=\lambda(\{x \in [r,r+2^N)\mid (p(x)_n)_{n< N+1}=(a_n)_{n< N+1}\})$. Let $(b_n)_{n< N+1}$ be defined by $b_N=1-a_N$ and $b_n=a_n$ for $0\leq n\leq N-1$. By Lemma \ref{parity2} we have $\Col_R^N(x+2^N)=\Col_R^N(x)+3^{\sum_{i=0}^{N-1}a_i}$ for all $x\in [r,r+2^N)$ such that $ (p(x)_n)_{n< N}=(a_n)_{n< N}$. Since $3^{\sum_{i=0}^{N-1}a_i}$ is odd, we get $p(x)_N=1-p(x+2^N)_N$ for all $x\in [r,r+2^N)$ with $(p(x)_n)_{n< N}=(a_n)_{n< N}$. Since the Lebesgue measure is invariant under translation, we obtain that  $\alpha=\lambda(\{x \in [r+2^N,r+2^{N+1})\mid (p(x)_n)_{n< N}=(b_n)_{n< N+1}\})$, $1-\alpha=\lambda(\{x \in [r+2^N,r+2^{N+1})\mid (p(x)_n)_{n< N+1}=(a_n)_{n< N+1}\})$, and $1-\alpha=\lambda(\{x \in [r,r+2^{N})\mid (p(x)_n)_{n< N+1}=(b_n)_{n< N+1}\})$. Thus $\lambda(\{x \in [r,r+2^{N+1})\mid (p(x)_n)_{n< N+1}=(a_n)_{n< N+1}\})=\lambda(\{x \in [r,r+2^N)\mid (p(x)_n)_{n< N}=(a_n)_{n< N+1}\})+\lambda(\{x \in [r+2^N,r+2^{N+1})\mid (p(x)_n)_{n< N}=(a_n)_{n< N+1}\})=1$.
	\end{proof}
	\begin{defn}
		For any non-empty interval $I\subseteq(0,\infty)$ of finite length let $\mu_{I}$ denote the probability measure defined on $I$ by $\mu_I(A)=\frac{\lambda(A)}{\lambda(I)}$ for any Lebesgue-measurable set $A\subseteq I$.
	\end{defn}
	Reformulating the result we obtain:
	\begin{lem}\label{uniform}
		The push-forward measure of $\mu_{[r,r+2^N)}$ under the map $x\mapsto (p(x)_n)_{0\leq n<N}$ is the uniform measure on $\{0,1\}^{\{0,\cdots,N-1\}}$.
	\end{lem}
	The preceding lemma is crucial for what follows. We continue with the following lemma:
	\begin{lem}\label{Hoeffding}
		Suppose that $(a,b)\subseteq \R$ is a non-empty interval, $\epsilon>0$, and $0\leq N\leq \lceil\log_2(b-a)\rceil$. Then
		$\mu_{(a,b)}(\{x\in(a,b)\mid \sum_{k=0}^{N-1}p(x)_k\geq(\frac{1}{2}+\epsilon)\cdot N\})\leq 2e^{-2\epsilon^2N}$ and $\mu_{(a,b)}(\{x\in(a,b)\mid \sum_{k=0}^{N-1}p(x)_k\leq (\frac{1}{2}-\epsilon)\cdot N\})\leq 2e^{-2\epsilon^2N}$.
	\end{lem}
	\begin{proof}
		This is a consequence of Hoeffding's inequality (see \cite{Hoeffding}) which states that if $\nu_n$ is the uniform measure on $\{0,1\}^{\intervalN{0}{n}}$, or equivalently the $n$-fold product measure of the uniform distribution on $\{0,1\}$ for some $n\in\N$, then $$\nu_{n}(\{x\in\{0,1\}^{\intervalN{0}{n}}\mid\sum_{k=0}^{n-1}x_i\geq(\frac{1}{2}+\epsilon)n\})\leq e^{-2\epsilon^2n}$$
		and $$\nu_{n}(\{x\in\{0,1\}^{\intervalN{0}{n}}\mid\sum_{k=0}^{n-1}x_i\leq(\frac{1}{2}-\epsilon)n\})\leq e^{-2\epsilon^2n}.$$
		Set $M=\lceil\log_2(b-a)\rceil$. By Proposition \ref{uniform}  the push-forward measure of $\mu_{[a,a+2^M)}$ under the map $x\mapsto (p(x)_n)_{0\leq n<M}$ is the uniform measure on $\{0,1\}^{\intervalN{0}{M}}$ and thus the push-forward measure of $\mu_{[a,a+2^M)}$ under the map $x\mapsto (p(x)_n)_{0\leq n<N}$ is the uniform measure on $\{0,1\}^{\intervalN{0}{N}}$. Here and in the following we will not detail why the maps and sets under consideration are measurable. Hence 	$\mu_{(a,a+2^M)}(\{x\in(a,a+2^M)\mid \sum_{k=0}^{N-1}p(x)_k\geq (\frac{1}{2}+\epsilon)\cdot N\})\leq e^{-2\epsilon^2N}$ and $\mu_{(a,a+2^M)}(\{x\in(a,a+2^M)\mid \sum_{k=0}^{N-1}p(x)_k\leq (\frac{1}{2}-\epsilon)\cdot N\})\leq e^{-2\epsilon^2N}$. Or equivalently $\lambda(\{x\in(a,a+2^M)\mid \sum_{k=0}^{N-1}p(x)_k\geq (\frac{1}{2}+\epsilon)\cdot N\})\leq 2^Me^{-2\epsilon^2N}$ and $\lambda(\{x\in(a,a+2^M)\mid \sum_{k=0}^{N-1}p(x)_k\leq (\frac{1}{2}-\epsilon)\cdot N\})\leq 2^Me^{-2\epsilon^2N}$, thus also $\lambda(\{x\in(a,b)\mid \sum_{k=0}^{N-1}p(x)_k\geq(\frac{1}{2}+\epsilon)\cdot N\})\leq 2^Me^{-2\epsilon^2N}$ and $\lambda(\{x\in(a,a+2^M)\mid \sum_{k=0}^{N-1}p(x)_k\leq (\frac{1}{2}-\epsilon)\cdot N\})\leq 2^Me^{-2\epsilon^2N}$, since $(a,b)\subseteq(a,a+2^M)$. Thus $\mu_{(a,b)}(\{x\in(a,b)\mid \sum_{k=0}^{N-1}p(x)_k\geq(\frac{1}{2}+\epsilon)\cdot N\})\leq\frac{2^M}{b-a}e^{-2\epsilon^2N}\leq2e^{-2\epsilon^2N} $ and $\mu_{(a,b)}(\{x\in(a,b)\mid \sum_{k=0}^{N-1}p(x)_k\leq (\frac{1}{2}-\epsilon)\cdot N\})\leq\frac{2^M}{b-a}e^{-2\epsilon^2N}\leq 2e^{-2\epsilon^2N} $ since $b-a>2^{M-1}$.
	\end{proof}
	To show Theorem \ref{K} we are actually showing the stronger statement that the set of all $x$ such that the iterations of $3^nx$ under $\Col_R$ eventually reach a value less then some $K>0$ for all $n\in\N$ is Lebesgue-co-null. Before we proof this we introduce some helpful notation: 
	\begin{defn}
		Let $t(x)$ be the unique integer such that $3^{t(x)}x\in (\frac{3}{4},\frac{9}{4}]$ for $x\in(0,\infty)$.
		Let $\pi:(0,\infty)\rightarrow (\frac{3}{4},\frac{9}{4}],\ab x\mapsto 3^{t(x)}x$ the \textit{projection} of $(0,\infty)$ onto $(\frac{3}{4},\frac{9}{4}]$.
\end{defn}
	We gather some properties of $\pi$ in the following lemma:
	\begin{lem}\label{m}
		\begin{enumerate}
			\item For every $z\in\Z$ the restriction of $\pi$ to $3^z(\frac{3}{4},\frac{9}{4}]$ is a homeomorphism from $3^z(\frac{3}{4},\frac{9}{4}]$ to $(\frac{3}{4},\frac{9}{4}]$.
			\item $\pi\circ \Col_R=\Col_R\circ \pi$.
			\item We have $T^n(x)=\frac{3^m}{2^n}x$ for every $x\in (\frac{3}{4},\frac{9}{4}]$ for some $m$ such that $n\log_3(2)-1< m < 1+ n\log_3(2)$.
		\end{enumerate}
	\end{lem}
	\begin{proof}
	$1.$ and $2.$ are immediate from the definition.	To see $3.$ note that $T^n(x)=\frac{3^m}{2^n}x$ for some $m\leq n$ by definition of $T$ and $T^n(x)\in (\frac{3}{4},\frac{9}{4}]$. 
	Now $x,\frac{3^m}{2^n}x\in(\frac{3}{4},\frac{9}{4}]$ imply that $\frac{3}{4}<\frac{3^m}{2^n}x\leq \frac{9}{4}$ and $\frac{3}{4}<x\leq \frac{9}{4}$ and we obtain $\frac{3^m}{2^n}\frac{3}{4}<\frac{9}{4}$ and $\frac{3}{4}<\frac{3^m}{2^n}\frac{9}{4}$, thus $\frac{1}{3}<\frac{3^m}{2^n}<3$, or equivalently $n\log_3(2)-1<m<1+n\log_3(2)$.
	\end{proof}

	\begin{thm}\label{Bound}
		There exists $K>0$ such that   $\mu_{(\frac{3}{4},\frac{9}{4}]}(\{x\in(\frac{3}{4},\frac{9}{4}]\mid \forall N\in\N:\ab \liminf_{n\rightarrow\infty} \Col_R^n(3^Nx)\leq K\})=1$. 
	\end{thm}
	\begin{proof}
		Note that the push-forward measure of $\mu_{(a,b)}$ under multiplication with $r>0$ is $\mu_{(ra,rb)}$. Choose $\epsilon>0$ such that $2>3^{\frac{1}{2}+\epsilon}$ and $b>2$. Define  $N_0=\lfloor \frac{\log_3 3b}{\frac{1}{2}+\epsilon}\rfloor$. For  $N\geq N_0$ consider the interval $(0,b\cdot \frac{2^N}{3^{\lfloor(\frac{1}{2}+\epsilon)\cdot N)\rfloor}}]  
		$, set $M_N=\lfloor\log_2(b\cdot \frac{2^N}{3^{\lfloor(\frac{1}{2}+\epsilon)\cdot N)\rfloor}})\rfloor$ and the set $A_N=\{x\in(0,b\cdot \frac{2^N}{3^{\lfloor(\frac{1}{2}+\epsilon)\cdot N)\rfloor}}]\mid \sum_{k=0}^{M_N-1}p(x)_k\geq(\frac{1}{2}+\epsilon)\cdot M_N\}$. Then $$\mu_{(0,b)}(\frac{3^{\lfloor\frac{1}{2}+\epsilon)\cdot N)\rfloor}}{2^N}A_n)\leq 2e^{-2\epsilon^2M_N}\leq 2e^{2\epsilon^2}e^{-2\epsilon^2\log_2(b\cdot (\frac{2^N}{3^{1+N(\frac{1}{2}+\epsilon)}}))}=2e^{2\epsilon^2}e^{-2\epsilon^2(\log_2(\frac{b}{3})+N\log_2( \frac{2}{3^{\frac{1}{2}+\epsilon}}))}.$$
		Now we can choose $b\geq 2$ large enough so that $\sum_{N=N_0}^\infty2e^{2\epsilon^2}e^{-2\epsilon^2(\log_2(\frac{b}{3})+N\log_2( \frac{2}{3^{\frac{1}{2}+\epsilon}}))}<1$. This means $A=\bigcup_{N\geq N_0}(\frac{3^{\lfloor(\frac{1}{2}+\epsilon)N\rfloor}}{2^N})A_N$ has $\mu_{(0,b)}$-measure $<1$ and thus $B=(0,b)\setminus A$ has positive $\mu_{(0,b)}$-measure.
		
		\begin{sublem}\label{sublemma}
			Suppose that $x\in B$ and $N\in\N$, then $\Col_R^N(\frac{2^N}{3^{\lfloor(\frac{1}{2}+\epsilon)\cdot N\rfloor}}\cdot x)\preccurlyeq_3 3^{1+\lfloor (\frac{1}{2}-\epsilon)\cdot N_0\rfloor}\cdot x$. In particular, $\Col_R^N(\frac{2^N}{3^{\lfloor(\frac{1}{2}+\epsilon)\cdot N\rfloor}}\cdot x)\leq (3b)^{\frac{2}{1-2\epsilon}}$.
		\end{sublem}
		\begin{proof}[Proof of sublemma]

			We proceed by induction: If $N\leq N_0$, then $\Col_R^N(\frac{2^N}{3^{\lfloor(\frac{1}{2}+\epsilon)\cdot N)\rfloor}}\cdot x)\preccurlyeq_3 \frac{3^N}{2^N}\frac{2^N}{3^{\lfloor(\frac{1}{2}+\epsilon)\cdot N)\rfloor}}\cdot x\preccurlyeq_3{3^{1+\lfloor (\frac{1}{2}-\epsilon)\cdot N_0)\rfloor}}\cdot x.$ 
			Now suppose that $N\geq \frac{\log_3 3b}{\frac{1}{2}+\epsilon}$.  This implies that $M_N\leq N$, since if towards a contradiction  $M_N>N$, i.e., $\lfloor\log_2(b\cdot \frac{2^N}{3^{\lfloor(\frac{1}{2}+\epsilon)\cdot N\rfloor}})\rfloor>N$, then also $\log_2(b\cdot \frac{2^N}{3^{\lfloor(\frac{1}{2}+\epsilon)\cdot N\rfloor}})>N$, thus $b\cdot \frac{2^N}{3^{\lfloor(\frac{1}{2}+\epsilon)\cdot N)\rfloor}}>2^N$, or $b>3^{\lfloor(\frac{1}{2}+\epsilon)\cdot N\rfloor}$, thus $\log_3 b>\lfloor(\frac{1}{2}+\epsilon)\cdot N\rfloor\geq (\frac{1}{2}+\epsilon)\cdot N-1$, thus $N<\frac{\log_3 3b}{\frac{1}{2}+\epsilon}$, a contradiction.\\ Since $x\in B$ it follows that $\frac{2^N}{3^{\lfloor(\frac{1}{2}+\epsilon)\cdot N\rfloor}}\cdot x\in (\frac{2^N}{3^{\lfloor(\frac{1}{2}+\epsilon)\cdot N\rfloor}}(0,b])\setminus A_N$, thus by Lemma \ref{parity} we conclude $$\Col_R^{M_N}(\frac{2^N}{3^{\lfloor(\frac{1}{2}+\epsilon)\cdot N\rfloor}}\cdot x)=\frac{3^{\sum_{k=0}^{M_N-1}p(\frac{2^N}{3^{\lfloor(\frac{1}{2}+\epsilon)\cdot N\rfloor}}\cdot x)_k}}{2^{M_N}}\cdot \frac{2^N}{3^{\lfloor(\frac{1}{2}+\epsilon)\cdot N\rfloor}}x\preccurlyeq_3 \frac{3^{\lfloor(\frac{1}{2}+\epsilon)\cdot M_N\rfloor}}{2^{M_N}}\cdot \frac{2^N}{3^{\lfloor(\frac{1}{2}+\epsilon)\cdot N\rfloor}}x.$$ Since ${\lfloor(\frac{1}{2}+\epsilon)\cdot (N-M_N)\rfloor}+j= {\lfloor(\frac{1}{2}+\epsilon)\cdot N\rfloor-\lfloor(\frac{1}{2}+\epsilon)\cdot M_N\rfloor}$ for some $j\in\{0,1\}$, it follows that $$\frac{3^{\lfloor(\frac{1}{2}+\epsilon)\cdot M_N\rfloor}}{2^{M_N}}\cdot\frac{2^N}{3^{\lfloor(\frac{1}{2}+\epsilon)\cdot N\rfloor}}x=\frac{2^{N-M_N}}{3^{\lfloor(\frac{1}{2}+\epsilon)\cdot N\rfloor-\lfloor(\frac{1}{2}+\epsilon)\cdot M_N\rfloor}}x=\frac{2^{N-M_N}}{3^{\lfloor(\frac{1}{2}+\epsilon)\cdot (N-M_N)\rfloor+j}}y\preccurlyeq_3\frac{2^{N-M_N}}{3^{\lfloor(\frac{1}{2}+\epsilon)\cdot (N-M_N)\rfloor}}x.$$ 
			Thus $$\Col_R^{M_N}(\frac{2^N}{3^{\lfloor(\frac{1}{2}+\epsilon)\cdot N\rfloor}}\cdot x)\preccurlyeq_3\frac{2^{N-M_N}}{3^{\lfloor(\frac{1}{2}+\epsilon)\cdot (N-M_N)\rfloor}}x.$$
			Since $M_N\leq N$, we conclude that $0\leq N-M_N$ and since also $M_N=\lfloor\log_2(b\cdot \frac{2^N}{3^{\lfloor(\frac{1}{2}+\epsilon)\cdot N\rfloor}})\rfloor\geq 1$, since $b\geq2$ and $ 2^N>3^{\lfloor(\frac{1}{2}+\epsilon)\cdot N\rfloor}$ we conclude $0\leq N-M_N<N$. 
			By induction hypothesis we obtain $$\Col_R^{N-M_N}(\frac{2^{N-M_N}}{3^{\lfloor(\frac{1}{2}+\epsilon)\cdot (N-M_N)\rfloor}}x)\preccurlyeq_3{3^{1+\lfloor (\frac{1}{2}-\epsilon)\cdot N_0\rfloor}}\cdot x.$$ We have just shown that $\Col_R^{M_N}(\frac{2^N}{3^{\lfloor(\frac{1}{2}+\epsilon)\cdot N)\rfloor}}\cdot x)\preccurlyeq_3\frac{2^{N-M_N}}{3^{\lfloor(\frac{1}{2}+\epsilon)\cdot (N-M_N)\rfloor}}x$, thus $$\Col_R^{N}(\frac{2^N}{3^{\lfloor(\frac{1}{2}+\epsilon)\cdot N)\rfloor}}\cdot x)=\Col_R^{N-M_N}(\Col_R^{M_N}(\frac{2^N}{3^{\lfloor(\frac{1}{2}+\epsilon)\cdot N))\rfloor}}\cdot x))\preccurlyeq_3\Col_R^{N-M_N}(\frac{2^{N-M_N}}{3^{\lfloor(\frac{1}{2}+\epsilon)\cdot (N-M_N)\rfloor}}x))\preccurlyeq_3{3^{1+\lfloor (\frac{1}{2}-\epsilon)\cdot N_0\rfloor}}\cdot x.$$
			The second claim follows from the fact, that $3^{1+\lfloor (\frac{1}{2}-\epsilon)\cdot N_0)\rfloor}\cdot x\leq 3^{1+\lfloor (\frac{1}{2}-\epsilon)\cdot N_0)\rfloor}\cdot b= 3^{1+\lfloor (\frac{1}{2}-\epsilon)\cdot \lfloor \frac{\log_3 3b}{\frac{1}{2}+\epsilon}\rfloor\rfloor}\cdot b\leq 3\cdot 3^{ (\frac{1}{2}-\epsilon)\cdot \frac{\log_3 3b}{\frac{1}{2}+\epsilon}}\cdot b=3\cdot (3b)^{\frac{1-2\epsilon}{1+2\epsilon}}b=(3b)^{\frac{2}{1-2\epsilon}}$.
		\end{proof}
		 To finish the proof note that $\pi(B\cap (\frac{3^{M+1}}{4},\frac{3^{M+2}}{4}])$ has positive $\mu$-measure for some $M\in\Z$. Set $C=\pi(B\cap (\frac{3^{M+1}}{4},\frac{3^{M+2}}{4}])$ for such a $M$. Since $T$ is uniquely ergodic and $C$ is $\mu$-positive, the set $D=\bigcup_{z\in\Z}T^z(C)$ has $\mu$-measure $1$. Consider the sets $F_z=\{x\in(\frac{3}{4},\frac{9}{4}]\mid T^z(x)\in C\wedge \forall w>z: T^w(x)\notin C \}$. Then the $F_z$ are pairwise disjoint and $T(F_z)=F_{z-1}$ thus $\mu(F_z)=\mu(F_w)$ for all $z,w\in\Z$, thus necessarily $\mu(F_z)=0$ for all $z\in\Z$, where $\mu$ denotes the $T$-invariant probability measure on $(\frac{3}{4},\frac{9}{4}]$.  Thus the set $E=\{x\in (\frac{3}{4},\frac{9}{4}]\mid \#\{n\in \N\mid T^n(x)\in C\}=\infty\}$
		  is of $\mu_{(\frac{3}{4},\frac{9}{4}]}$-measure $1$. Now, let $x\in E$. Then there exist arbitrarily large $N\in\N$ such that and $T^N(x)\in C$. Thus $3^M(T^N(x))\in B$. By Lemma \ref{m} we have $2^NT^N(x)=3^mx$ for some $N\log_3(2)-1< m < 1+ N\log_3(2)$ thus $3^M(T^N(x))=\frac{3^{m+M}}{2^N}x$.
		Now, to finish the proof, set $K=(3b)^{\frac{2}{1-2\epsilon}}$ and for a given $n\in\N$ choose $N$ large enough such that $n<M+\lfloor N\log_3(2)-1\rfloor-\lfloor(\frac{1}{2}+\epsilon)\cdot N\rfloor\leq M+m-\lfloor(\frac{1}{2}+\epsilon)\cdot N\rfloor$. Then by the sublemma $\Col_R^N(3^nx)\preccurlyeq_3\Col_R^N(3^{M+m-\lfloor(\frac{1}{2}+\epsilon)\cdot N\rfloor}x)=\Col_R^N(\frac{2^N}{3^{\lfloor(\frac{1}{2}+\epsilon)\cdot N)\rfloor}}\frac{3^{m+M}}{2^N}x)\leq (3b)^{\frac{2}{1-2\epsilon}}=K$, thus $\Col_R^N(3^nx)\leq K$ and since $N$ can be chosen arbitrarily large we conclude $\liminf_{N\rightarrow\infty}\Col_R^N(3^nx)\leq K$.
	\end{proof}

	\begin{rem}\label{test}
		The obvious question is whether almost-all orbits go all the way down, i.e., whether we can chose $K=\frac{3}{4}$. In the following we detail the empirical argument given in the introduction that Theorem \ref{Bound} is very likely true for $K=\frac{3}{4}$. In the proof of \ref{Bound} we constructed for each $\epsilon>0$ and $b>2$ a set $B\subseteq(0,b)$ and $N_0=\lfloor \frac{\log_3 3b}{\frac{1}{2}+\epsilon}\rfloor$. Set $\alpha=\mu_{(0,b)}(B)$.  By Sublemma \ref{sublemma} we know that  $\Col_R^N(\frac{2^N}{3^{\lceil(\frac{1}{2}+\epsilon)\cdot N\rceil}}x) \preccurlyeq_3 3^{1+\lfloor (\frac{1}{2}-\epsilon)\cdot N_0\rfloor}\cdot x$ for all $N\in\N$ and $x\in B$. If for some $M>0$ the set $A_M=\{x\in(0,b)\mid \Col_R^M( 3^{1+\lfloor (\frac{1}{2}-\epsilon)\cdot N_0\rfloor}\cdot x)\leq\frac{9}{4}$ has $\mu_{(0,b)}$-measure greater than $1-\alpha$, then $\mu_{(0,b)}(A\cap B)>0$. Furthermore, if $x\in A\cap B$ then $\Col_R^{N+M}(\frac{2^N}{3^{\lceil(\frac{1}{2}+\epsilon)\cdot N\rceil}}x) \leq \frac{9}{4}$ for all $N\in\N$ and we can use $A\cap B$ in the proof of Theorem \ref{Bound} to obtain $\frac{9}{4}$ and thus also $\frac{3}{4}$ as a bound since every orbit in $(\frac{3}{4},\frac{9}{4}]$ is dense by Proposition \ref{T}. For suitable values of $b,\epsilon$ and $M$ one may test empirically if $A_M$ has measure less or equal than $1-\alpha$ by randomly choosing long enough initial segments of $x\in (0,b)$ so that performing $\Col_R^M$ only depends on the initial segment of $x$. A test with $b=3^{200}$, $\epsilon=0.13$, $M=6000$ and initial segment of $x$ with length $6500$ was performed on a computer by Claudius Röhl for $3000$ repetitions $(x_i)_{0\leq i\leq 2999}$. The corresponding $\alpha$ is greater then $0.1$. All $3000$ trials had finite stopping time $\tau_{\frac{9}{4}}(x_i)<6000$. Thus if $A_M$ has measure less or equal than $0.9$ this outcome has a probability of less than $0.9^{3000}$ which is less than $10^{-137}$ providing strong empirical evidence that Conjecture \ref{verylikely} is true.
	\end{rem}
\begin{rem}
	If Theorem \ref{K} does not hold for $K\leq\frac{9}{4}$, then necessarily $\limsup_{n\rightarrow\infty}\Col_R^n(x)=\infty$ for a Lebesgue co-null set of $x\in(0,\infty)$, since if\ab\ab $\limsup_{n\rightarrow\infty}\Col_R^n(x)<\infty$, then\ab $\limsup_{n\rightarrow\infty}(x)\Col_R^n=\frac{9}{4}$ by a argument similar to that in Theorem \ref{comeager}. 
\end{rem}

As a corollary we prove Theorem \ref{K}:
\begin{proof}[Proof of Theorem \ref{K}]
By Theorem \ref{Bound} we find $K>0$ such that $$\mu_{(\frac{3}{4},\frac{9}{4}]}(\{x\in(\frac{3}{4},\frac{9}{4}]\mid \forall N\in\N:\ab \liminf_{n\rightarrow\infty} (\{\Col_R^n(3^Nx)\mid n\in\N\})\leq K\})=1.$$ Set $A=\{x\in(\frac{3}{4},\frac{9}{4}]\mid \forall N\in\N:\ab \liminf_{n\rightarrow\infty} \Col_R^n(3^Nx)\leq K\}$ and define $B=\bigcup_{n\in\N}3^nA$, then clearly $B$ is Lebesgue-co-null in $(\frac{3}{4},\infty)$ and if $x\in B$ then by definition of $A$ we get that $\liminf_{n\rightarrow\infty} \Col_R^n(x)\leq K$.
\end{proof}
	We outsource the following technical lemma from the following proof:
	\begin{lem}\label{partiondensity}
		Let $S\subseteq (0,\infty)$ and $a_n\in(0,\infty)$ an increasing diverging sequence such that there exits a bound $q>0$ with $\frac{a_{n}}{a_{n+1}}\geq q$ for all $n\in\N$ and suppose that $\epsilon>0$. If \ab $\liminf_{n\rightarrow\infty}\mu_{(a_{n},a_{n+1})}(S\cap (a_{n},a_{n+1}))\geq1-\epsilon$ then $\liminf_{R\rightarrow\infty}\mu_{(0,R)}(S\cap (0,R))\geq1-\frac{\epsilon}{q}$.
	\end{lem}
	\begin{proof}
		
	 Look at the set $U=(0,\infty)\setminus S$ and choose $\delta,\eta$ such that $\epsilon<\delta<\eta$. We can find $n_0\in\N$ such that $\mu_{(a_{n+1},a_n)}(U\cap (a_{n+1}-a_n))\leq \delta$ for every $n\geq n_0$. Let $R>a_{n_0}$ and find $n\in\N$ such that $a_n<R\leq a_{n+1}$. Then $\lambda(U\cap(0,R))\leq \lambda(U\cap (0,a_{n_0}))+\sum_{k=n_0}^{n}\lambda(U\cap(a_n,a_{n+1}))\leq \lambda(U\cap (0,a_{n_0}))+\delta\sum_{k=n_0}^{n}(a_{k+1}-a_k)=\lambda(U\cap (0,a_{n_0}))+\delta (a_{n+1}-a_{n_0})$.
Thus 	 $\mu_{(0,R)}(U\cap(0,R))\leq \frac{\lambda(U\cap (0,a_{n_0}))}{R}+\delta \frac{(a_{n+1}-a_{n_0})}{a_n}\leq \frac{\lambda(U\cap (0,a_{n_0}))}{R}-\delta\frac{a_{n_0}}{a_n}+\frac{\delta}{q}<\frac{\eta}{q}$, when $R$ is sufficiently large. Since $\eta>\epsilon$ was arbitrary the claim follows.
	\end{proof}
 
	\begin{thm}\label{ratetomin}

		There exists $K>0$ such that the set $\{x\in(\frac{3}{4},\infty)\mid \min_{n\leq \log_2(x)(\frac{1}{1-\frac{\log_2(3)}{2}}+\epsilon)}\Col_R^n(x)\leq K\}$ has real density $1$ for every $\epsilon>0$.
		
	\end{thm}
	\begin{proof}
		Consider - as in the proof of Theorem \ref{Bound} - the sets $A_{N,c,\delta}=\{x\in(0,c\cdot \frac{2^N}{3^{\lfloor(\frac{1}{2}+\delta)\cdot N\rfloor}}]\mid \sum_{k=0}^{M_N-1}p(x)_k>(\frac{1}{2}+\delta)\cdot M_N\}$, where $M_N=\lfloor\log_2(c\cdot \frac{2^N}{3^{\lfloor(\frac{1}{2}+\delta)\cdot N\rfloor}})\rfloor$. We also set $N_0=\lfloor \frac{\log_3 3c}{\frac{1}{2}+\delta}\rfloor$.
		
		Then $$\mu_{(0,c)}(\frac{3^{\lfloor(\frac{1}{2}+\delta)\cdot N\rfloor}}{2^N}A_{N,c,\delta})< 2e^{-2\delta^2M_N}\leq 2e^{2\delta^2}e^{-2\delta^2\log_2(c\cdot (\frac{2^N}{3^{1+N(\frac{1}{2}+\delta)}}))}=2e^{2\delta^2}e^{-2\delta^2(\log_2(\frac{c}{3})+N\log_2( \frac{2}{3^{\frac{1}{2}+\delta}}))}.$$
		
		Now for every $\gamma>0$ we can choose $c$ large enough so that $\sum_{N=0}^\infty2e^{2\delta^2}e^{-2\delta^2(\log_2(\frac{c}{3})+N\log_2( \frac{2}{3^{\frac{1}{2}+\delta}}))}<\gamma$. By Theorem \ref{K} we can find $K>0$ and $M\in\N$ such that the set $D=\{x\in(0,c]\mid\min_{n\leq M}\Col_R^n(3^{1+\lfloor (\frac{1}{2}-\delta)\cdot N_0)\rfloor}\cdot x)> K \}$ has $\mu_{(0,c)}$-measure less than $\gamma$. Thus the set $F=D\cup\bigcup_{N\in\N}\frac{3^{\lfloor(\frac{1}{2}+\epsilon)\cdot N\rfloor}}{2^N}A_N$ is of $\mu_{(0,c)}$-measure less than $2\gamma$. Thus also the sets $F_N=\frac{2^N}{3^{\lfloor(\frac{1}{2}+\delta)\cdot N\rfloor}}F$ have $\mu_{(0,\frac{2^N}{3^{\lfloor(\frac{1}{2}+\delta)\cdot N\rfloor}}c)}$-measure less than $2\gamma$. Abbreviate $b_N=\frac{2^N}{3^{\lfloor(\frac{1}{2}+\delta)\cdot N\rfloor}}c.$
		Define  $H_N=(b_{N-8},b_N)\cap F_N$. $H_N$ has $\lambda$-measure less than $2\gamma\frac{2^N}{3^{\lfloor(\frac{1}{2}+\delta)\cdot N\rfloor}}c $, thus $$\mu_{(b_{N-8},b_N)}(H_N)
		<\frac{2\gamma\frac{2^N}{3^{\lfloor(\frac{1}{2}+\delta)\cdot N\rfloor}}c  }{b_N-b_{N-8}}
		=\frac{2^9\gamma }{2^8-{3^{\lfloor(\frac{1}{2}+\delta)\cdot N\rfloor-\lfloor(\frac{1}{2}+\delta)\cdot (N-8)\rfloor}}}
		\leq \frac{2^9\gamma }{2^8-3^5}, $$ if we choose $\delta>0$ small enough such that $\lfloor(\frac{1}{2}+\delta)\cdot N\rfloor-\lfloor(\frac{1}{2}+\delta)\cdot (N-8)\rfloor\leq 5$, which will be the case if $(\frac{1}{2}+\delta)\cdot 8\leq 5 $.\\
		Now, if $y\in I_N= (b_{N-8},b_N)\setminus H_N$, then by Sublemma \ref{sublemma} we have $\Col_R^N(y)\preceq_3 3^{1+\lfloor (\frac{1}{2}-\delta)\cdot N_0\rfloor}\frac{3^{\lfloor(\frac{1}{2}+\delta)\cdot N\rfloor}}{2^N}y$. Since $\frac{3^{\lfloor(\frac{1}{2}+\delta)\cdot N\rfloor}}{2^N}y\in(0,c]\setminus D$ we conclude $\min_{n\leq M+N}\Col_R^n(y)\leq K$. Now, $\log_2(y)\geq\log_2(\frac{2^{N-8}}{3^{\lfloor(\frac{1}{2}+\delta)\cdot (N-8)\rfloor}}c)\geq \log_2(\frac{c}{3})-8(1-(\frac{1}{2}+\delta)\log_2(3))+N(1-(\frac{1}{2}+\delta)\log_2(3))$. Thus $$N+M\leq \log_2(y)(\frac{1-\frac{\log_2(\frac{c}{3})-8(1-(\frac{1}{2}+\delta)\log_2(3))}{\log_2(y)}}{1-(\frac{1}{2}+\delta)\log_2(3)}+\frac{M}{\log_2(y)}).$$
		For any given $\epsilon>0$ we can choose $\delta$ small enough such that $\frac{1}{1-(\frac{1}{2}+\delta)\log_2(3)}<\frac{1}{1-\frac{1}{2}\log_2(3)}+\epsilon$. Then for $y$ sufficiently large we also have $N+M\leq \log_2(y)(\frac{1-\frac{\log_2(\frac{c}{3})-8(1-(\frac{1}{2}+\delta)\log_2(3))}{\log_2(y)}}{1-(\frac{1}{2}+\delta)\log_2(3)}+\frac{M}{\log_2(y)})<\log_2(y)(\frac{1}{1-\frac{1}{2}\log_2(3)}+\epsilon)$. Thus we have shown that $ \min_{n\leq \log_2(x)(\frac{1}{1-\frac{\log_2(3)}{2}}+\epsilon)}\Col_R^n(x)\leq K$  for all $x\in I_{8N}$ and sufficiently large $N$. Furthermore, $\mu_{(b_{8(N-1)},b_{8N})}(F_{8N})\geq1- \frac{2^9\gamma }{2^8-3^5}$. Note that $\frac{\frac{2^{8N}}{3^{\lfloor(\frac{1}{2}+\delta)\cdot 8N\rfloor}}c}{\frac{2^{8N+8}}{3^{\lfloor(\frac{1}{2}+\delta)\cdot (8N+8)\rfloor}}c}=\frac{3^{\lfloor(\frac{1}{2}+\delta)\cdot (8N+8)\rfloor}}{2^{8}3^{\lfloor(\frac{1}{2}+\delta)\cdot 8N\rfloor}}\geq\frac{3^{(\frac{1}{2}+\delta)\cdot (8N+8)-1}}{2^{8}3^{(\frac{1}{2}+\delta)\cdot 8N}}= \frac{3^{(\frac{1}{2}+\delta)8-1}}{2^{8}}$. Applying Lemma \ref{partitiondense} with $a_N=b_{8N}=\frac{2^{8N}}{3^{\lfloor(\frac{1}{2}+\delta)\cdot 8N\rfloor}}c$ we obtain $\liminf_{R\rightarrow\infty}\mu_{(0,R)}(\{x\in(\frac{3}{4},\infty)\mid \min_{n\leq \log_2(x)(\frac{1}{1-\frac{\log_2(3)}{2}}+\epsilon)}\Col_R^n(x)\leq K\})\geq 1-\frac{2^{8}}{3^{(\frac{1}{2}+\delta)8-1}}\frac{2^9\gamma }{2^8-3^5}$. As $\gamma>0$ was arbitrary, the claim follows.
	\end{proof}
	Note that in case $K=\frac{9}{4}$ is a bound in Theorem \ref{ratetomin}, then since once an orbit enters $(\frac{3}{4},\frac{9}{4}]$ it never leaves it again, Theorem \ref{ratetomin} takes the following form:
	\begin{thm}
		Suppose that $\lambda(\{x\in(\frac{3}{4},\infty)\mid \inf_{n\in\N}\Col_R^n(x)>\frac{9}{4}\})=0$.
		Then the set $\{x\in(\frac{3}{4},\infty)\mid  \Col_R^{\lfloor\log_2(x)(\frac{1}{1-\frac{\log_2(3)}{2}}+\epsilon)\rfloor}\leq \frac{9}{4}\}$ has real density $1$ for every $\epsilon>0$.
	\end{thm}

	\section{An approximation of the orbits of $\Col_R$}
	We begin by gathering some easy to verify properties concerning the notion of \realstardensity \ab (see Definition \ref{density}).
		\begin{lem}\label{realstardense}
		Suppose that $S_i\subseteq\N$ have \realdensity{A_i,B_i} for some $A_i>0$ and $0<B_i<1$ for $i\in\{0,1\}$. 
		\begin{enumerate}
			\item The set $S_0\cap S_1$ has \realdensity{A_0+A_1,\min\{B_0,B_1\}},
			\item If $S,T$ are \realstardense\ab then $S\cap T$ is \realstardense,
			\item $S_0$ has \realdensity{A,B} for every $A\geq A_0$ and $0<B\leq B_0$,
			\item $S_0\setminus(0,K)$ is \realstardense\ab  for every $K\in(0,\infty)$,
			\item Any set containing $S_0$ has \realdensity{A_0,B_0},
			\item $f\cdot S_0\subseteq \N$ is \realstardense\ab for every $f\in(0,\infty)$.
		\end{enumerate}
	\end{lem}
	\begin{proof}
		We only proof the last item. By assumption $\frac{\lambda((0,R)\setminus S_0)}{R}\leq \frac{A_0}{R^{B_0}}$, thus $\frac{\lambda((0,f\cdot R)\setminus f\cdot S_0)}{f\cdot R}\leq \frac{A_0}{R^{B_0}}$. Thus $\frac{\lambda((0,R)\setminus f\cdot S_0)}{R}\leq \frac{f^{B_0}A_0}{R^{B_0}}$, thus $f\cdot S$ has \realdensity{f\cdot A_0,B_0}. 
	\end{proof}
	\begin{lem}\label{partitiondense}
		Let $S\subseteq (0,\infty)$ and $a_n\in(0,\infty)$ an increasing diverging sequence such that there exits a bound $q>0$ with $\frac{a_{n}}{a_{n+1}}\geq q$ for all $n\in\N$. If there exists $C>0$ and $0<D<1$ such that $\frac{\lambda(S\cap (a_{n+1}-a_n))}{a_{n+1}-a_n}\geq 1-\frac{C}{a_{n+1}^D}$ then $S$ is \realstardense.
		
		If there exists $r>1$ such that $r\leq \frac{a_{n+1}}{a_n}$ for all but finitely many $n\in\N$, then the converse holds as well.
	\end{lem}
	\begin{proof}
		
		Suppose that there exist $C>0$ and $0<D<1$ such that $\frac{\lambda(S\cap (a_{n+1}-a_n))}{a_{n+1}-a_n}\geq 1-\frac{C}{a_{n+1}^D}$ for all $n\in\N$. Let $R>0$. Look at the set $U=(0,\infty)\setminus S$. We know that $\lambda(U\cap (a_{n+1},a_n)) \leq (a_{n+1}-a_n)\frac{C}{a_{n+1}^D}$. Let $n\in\N$ such that $a_n\leq R<a_{n+1}$. Then  $\lambda(U\cap (0,R)) \leq a_0+\sum_{i=0}^{n}(a_{i+1}-a_i)\frac{C}{(a_{i+1})^D}\leq a_0+ \int_{a_0}^{a_{n+1}}\frac{C}{x^{D}}dx\leq a_0+\frac{C}{1-D}a_{n+1}^{1-D}\leq a_0+\frac{C}{q^{1-D}(1-D)}R^{1-D}$, since $qa_{n+1}\leq a_n\leq R$. Thus  $\frac{\lambda(S\cap (0,R))}{R} \geq 1-\frac{a_0}{R}-\frac{C}{q^{1-D}(1-D)R^D}$. Thus, $S$ is \realstardense.
		
		Suppose now that $S$ is \realstardense\ab and there exists $r>1$ such that $\frac{a_{n+1}}{a_n}\geq r$ for all but finitely many $n\in\N$. Let $n\in\N$ be sufficiently large. There exist $C>0$ and $0<D<1$ such that $S$ has \realdensity{C,D}. Look at the set $U=(0,\infty)\setminus S$. Then $\frac{\lambda(U\cap (0,a_{n+1}))}{a_{n+1}}\leq \frac{C}{a_{n+1}^D}$. Thus $\frac{\lambda(U\cap (a_n,a_{n+1}))}{a_{n+1}-a_n}\leq \frac{a_{n+1}}{a_{n+1}-a_n}\frac{C}{a_{n+1}^D}=(1-\frac{a_{n}}{a_{n+1}})^{-1}\frac{C}{a_{n+1}^D}\leq \frac{r}{r-1}\frac{C}{a_{n+1}^D}$, thus $\frac{\lambda(S\cap (a_n,a_{n+1}))}{a_{n+1}-a_n}\geq 1- \frac{r}{r-1}\frac{C}{a_{n+1}^D}$.
	\end{proof}

	\begin{lem}\label{IterationR}
		If $S\subset\R$ is \realstardense\ab and $\zeta>0$, then the set $$\{x\in\R\mid \exists 0\leq l_0,l_1\leq 2\zeta \log_2x : \{3^{-l_0}\Col_R^{\lfloor\log_2x\rfloor}(x),3^{l_1}\Col_R^{\lfloor\log_2x\rfloor}(x)\}\subseteq S\}$$ is \realstardense.
	\end{lem}
	\begin{proof}
		We start with $l_0$. Set $a_n=2^{4n}$. Look at the set $P_n=\{x\in(a_{n},a_{n+1})\mid \leq (\frac{1}{2}-\zeta)4n\leq \sum_{i=0}^{4n-1}p(x)_i\leq (\frac{1}{2}+\zeta)4n\}$. By Lemma \ref{Hoeffding} we can find $C>0$ and $0<D<1$ such that $\frac{\lambda(P_n)}{a_{n+1}-a_n}\geq 1-\frac{C}{a_{n+1}^D}$. Look at $b_n=3^{\lfloor 4n(\frac{1}{2}-\zeta)\rfloor}$. Then $\frac{b_{n+1}}{b_{n}}\geq \frac{3^{ (4n+4)(\frac{1}{2}-\zeta)}}{3^{ 4n(\frac{1}{2}-\zeta)+1}}=3^{ 4(\frac{1}{2}-\zeta)-1} >1$, since we can assume that $4(\frac{1}{2}-\zeta)-1>0$ by taking a smaller $\zeta>0$, which does not change the result. By Lemma \ref{realstardense} $S_F=\bigcap_{i\in F}f_iS$ is \realstardense\ab for any finite set $F\subseteq(0,\infty)$. We will later specify $F$.
		By assumption and Lemma \ref{partitiondense}\ab  we can then find $C_0>0$ and $0<D_0<1$ such that $\frac{\lambda(S_F\cap(b_n,b_{n+1}))}{b_{n+1}-b_n}\geq 1-\frac{C_0}{b_{n+1}^{D_0}}$, for all $n\in\N$. But then for $S^\prime_F=\{x\in(a_n,a_{n+1})\mid \frac{3^{\lfloor 4n(\frac{1}{2}-\zeta)\rfloor}}{2^{4n}}\cdot x\in S_F\}$ we also have $\frac{\lambda(S^\prime_F\cap(a_n,a_{n+1}))}{a_{n+1}-a_n}\geq 1-\frac{C_0}{b_{n+1}^{D_0}}$. Now $b_{n}=3^{\lfloor 4n(\frac{1}{2}-\zeta)\rfloor}\geq \frac{1}{3}\cdot 3^{4n(\frac{1}{2}-\zeta)}=\frac{1}{3}\cdot 2^{4n(\frac{1}{2}-\zeta)\log_23}$, thus $\frac{\lambda(S^\prime_F\cap(a_n,a_{n+1}))}{a_{n+1}-a_n}\geq 1-\frac{C_1}{a_{n+1}^{D_1}}$ for  $C_1=3^{D_0}C_0$ and $D_1=D_0\cdot (\frac{1}{2}-\zeta)\log_23$. Thus we can further find $C_2>0$ and $0<D_2<1$ such that $\frac{\lambda(S^\prime_F\cap P_n)}{a_{n+1}-a_n}\geq 1-\frac{C_2}{a_{n+1}^{D_2}}$. If $x\in S^\prime_F\cap P_n$ then $4n\leq \lfloor\log_2x\rfloor< 4n+4$, $\frac{3^{\lfloor4n(\frac{1}{2}-\zeta)\rfloor}}{2^{4n}}x\leq \Col_R^{4n}(x)=\frac{3^{\sum_{k=0}^{4n}p(x)_k}}{2^{4n}}x\leq \frac{3^{4n(\frac{1}{2}+\zeta)}}{2^{4n}}x$, and $\frac{3^{\lfloor4n(\frac{1}{2}-\zeta)\rfloor}}{2^{4n}}x\in S_F$. We have $ \Col_R^{\lfloor\log_2x\rfloor }(x)=\frac{3^j}{2^i}\Col_R^{4n}(x)$ for some $0\leq j\leq i\leq 3$. Thus there exists $0\leq l^\prime\leq 8n\zeta+1\leq 2\zeta \log_2x+1$ such that $ 3^{-l^\prime}\frac{2^i}{3^j}\Col_R^{\log_2 x}(x)\in S_F$.  So, if $F$ contains $\{\frac{2^i}{3^j}\mid 0\leq j\leq 3, -1\leq i\leq 3\}$, then we get  $0\leq l\leq 2\zeta \log_2x$ such that $3^{-l}\Col_R^{\log_2 x}(x)\in S$. By Lemma \ref{partitiondense} the result for $l_0$ follows.
		
		With minor adjustments we go through the same argument to obtain the part of the result concerning $l_1$:
		
		Put $a_n=2^{4n}$. Look at the set $P_n=\{x\in(a_{n},a_{n+1})\mid \leq (\frac{1}{2}-\zeta)4n\leq \sum_{i=0}^{k}p(x)_i\leq (\frac{1}{2}+\zeta)4n\}$. By Lemma \ref{Hoeffding} we find $C>0$ and $0<D<1$ such that $\frac{\lambda(P_n)}{a_{n+1}-a_n}\geq 1-\frac{C}{a_{n+1}^D}$. Look at $b_n=3^{\lfloor 4n(\frac{1}{2}+\zeta)\rfloor}$. Then $\frac{b_{n+1}}{b_{n}}\geq \frac{3^{ (4n+4)(\frac{1}{2}+\zeta)}}{3^{ 4n(\frac{1}{2}+\zeta)+1}}=3^{ 4(\frac{1}{2}+\zeta)-1} >1$, since  $4(\frac{1}{2}+\zeta)-1>0$. By Lemma \ref{realstardense} $S_F=\bigcap_{i\in F}f_iS$ is \realstardense\ab for any finite set $F\subseteq(0,\infty)$. We will later specify $F$.
		By assumption and Lemma \ref{partitiondense}\ab  we can then find $C_0>0$ and $0<D_0<1$ such that $\frac{\lambda(S_F\cap(b_n,b_{n+1}))}{b_{n+1}-b_n}\geq 1-\frac{C_0}{b_{n+1}^{D_0}}$, for all $n\in\N$. But then for $S^\prime_F=\{x\in(a_n,a_{n+1})\mid \frac{3^{\lfloor 4n(\frac{1}{2}+\zeta)\rfloor}}{2^{4n}}\cdot x\in S_F\}$ we also have $\frac{\lambda(S^\prime_F\cap(a_n,a_{n+1}))}{a_{n+1}-a_n}\geq 1-\frac{C_0}{b_{n+1}^{D_0}}$. Now $b_{n}=3^{\lfloor 4n(\frac{1}{2}+\zeta)\rfloor}\geq \frac{1}{3}\cdot 3^{4n(\frac{1}{2}+\zeta)}=\frac{1}{3}\cdot 2^{4n(\frac{1}{2}+\zeta)\log_23}$, thus $\frac{\lambda(S^\prime_F\cap(a_n,a_{n+1}))}{a_{n+1}-a_n}\geq 1-\frac{C_1}{a_{n+1}^{D_1}}$ for  $C_1=3^{D_0}C_0$ and $D_1=D_0\cdot (\frac{1}{2}+\zeta)\log_23$. Thus we can further find $C_2>0$ and $0<D_2<1$ such that $\frac{\lambda(S^\prime_F\cap P_n)}{a_{n+1}-a_n}\geq 1-\frac{C_2}{a_{n+1}^{D_2}}$. If $x\in S^\prime_F\cap P_n$ then $4n\leq \lfloor\log_2x\rfloor< 4n+4$, $\frac{3^{4n(\frac{1}{2}-\zeta)}}{2^{4n}}x\leq \Col_R^{4n}(x)=\frac{3^{\sum_{k=0}^{4n}p(x)_k}}{2^{4n}}x\leq 3\frac{3^{\lfloor4n(\frac{1}{2}+\zeta)\rfloor}}{2^{4n}}x$, and $\frac{3^{\lfloor4n(\frac{1}{2}+\zeta)\rfloor}}{2^{4n}}x\in S_F$. We have $ \Col_R^{\lfloor\log_2x\rfloor}(x)=\frac{3^j}{2^i}\Col_R^{4n}(x)$ for some $0< j\leq i\leq 3$. Thus there exists $0\leq l^\prime\leq 8n\zeta+1\leq 2\zeta \log_2x+1$ such that $ 3^{l^\prime}\frac{2^i}{3^j}\Col_R^{\log_2 x}(x)\in S_F$.  So, if $F$ contains $\{\frac{2^i}{3^j}\mid 0< j\leq 3, 0\leq i\leq 4\}$, then we get  $0\leq l\leq \zeta \log_2x$ such that $3^{l}\Col_R^{\log_2 x}(x)\in S$. By Lemma \ref{partitiondense} the result follows for $l_1$.
		
		Since the intersection of two \realstardense\ab sets is \realstardense\ab again, the result follows.
	\end{proof}
Instead of proving Theorem \ref{maintheoremdynamicintroductoin} we prove a slightly stronger variant of it, which allows to extend the domain of $k$ by $\theta\log_2x$ for some $\theta>0$.
	\begin{thm}\label{maintheoremdynamics}
		Suppose that $\epsilon>0$. Then there exits $\theta>0$ so that the set $\{x\in(0,\infty)\mid\forall k\leq (\frac{1}{1-\frac{\log_23}{2}}+\theta)\log_2x:  (\frac{3^{\frac{1}{2}}}{2})^kx^{1-\epsilon}\leq \Col_R^k(x)\leq (\frac{3^{\frac{1}{2}}}{2})^kx^{1+\epsilon}\}$ is \realstardense.
	\end{thm}
	\begin{proof}
		First we show that it is enough to show that for every $\epsilon>0$ the set  $S^\lambda_\epsilon=\{x\in(0,\infty)\mid\forall k\leq (\frac{1-\lambda}{1-\frac{\log_23}{2}})\log_2x:  (\frac{3^{\frac{1}{2}}}{2})^kx^{1-\epsilon}\leq \Col_R^k(x)\leq (\frac{3^{\frac{1}{2}}}{2})^kx^{1+\epsilon}\}$ is \realstardense\ab for all $1\geq\lambda> 0$.
		To see this suppose that $S^\lambda_\delta$ is \realstardense. Assume that $x\in S^\lambda_\delta$ and $$(\frac{1-\lambda}{1-\frac{\log_23}{2}})\log_2x<k\leq (\frac{1}{1-\frac{\log_23}{2}}+\theta)\log_2x.$$ Set $k_0=\lfloor(\frac{1-\lambda}{1-\frac{\log_23}{2}})\log_2x\rfloor$. Since $(\frac{3^{\frac{1}{2}}}{2})^{k_0}x^{1-\delta}\leq \Col_R^{k_0}(x)\leq (\frac{3^{\frac{1}{2}}}{2})^{k_0}x^{1+\delta}$, we obtain $$\Col_R^{k_0+(k-k_0)}(x)=\Col_R^{k-k_0}(\Col_R^{k_0}(x))\geq\frac{1}{2^{k-k_0}} (\frac{3^{\frac{1}{2}}}{2})^{k_0}x^{1-\delta}=\frac{1}{3^{\frac{1}{2}(k-k_0)}} (\frac{3^{\frac{1}{2}}}{2})^{k}x^{1-\delta}$$ and $$\Col_R^{k_0+(k-k_0)}(x)=\Col_R^{k-k_0}(\Col_R^{k_0}(x))\leq(\frac{3}{2})^{k-k_0} (\frac{3^{\frac{1}{2}}}{2})^{k_0}x^{1+\delta}={3^{\frac{1}{2}(k-k_0)}} (\frac{3^{\frac{1}{2}}}{2})^{k}x^{1+\delta}.$$ Thus $\frac{1}{3^{\frac{1}{2}(k-k_0)}} (\frac{3^{\frac{1}{2}}}{2})^{k}x^{1-\delta}\leq\Col_R^{k}(x)\leq{3^{\frac{1}{2}(k-k_0)}} (\frac{3^{\frac{1}{2}}}{2})^{k}x^{1+\delta}$. If we choose $\lambda,\theta,\delta>0$ small enough we can ensure that  $\frac{1}{3^{\frac{1}{2}(k-k_0)}} x^{1-\delta}\leq x^{1-\epsilon}$ and ${3^{\frac{1}{2}(k-k_0)}} x^{1+\delta}\leq x^{1+\epsilon}$. To see that this is possible just note that $\frac{k-k_0-1}{\log_2x}\leq \frac{1}{1-\frac{\log_23}{2}}+\theta- \frac{1-\lambda}{1-\frac{\log_23}{2}}$ goes to $0$ as $\theta,\lambda$ go to $0$. Thus, if $\delta<\epsilon$, any $x\in S^{\lambda}_\delta$ fulfills $\forall k\leq (\frac{1}{1-\frac{\log_23}{2}}+\theta)\log_2x:  (\frac{3^{\frac{1}{2}}}{2})^kx^{1-\epsilon}\leq \Col_R^k(x)\leq (\frac{3^{\frac{1}{2}}}{2})^kx^{1+\epsilon}$. Thus we are done showing that it suffices to show that the $S^\lambda_\epsilon$ are \realstardense.
		
		Set $a_n=2^n$ and define for any $\delta>0$ the set $A_n=\{x\in(a_{n},a_{n+1})\mid\forall k\leq n: (\frac{3^{\frac{1}{2}}}{2})^kx^{1-\delta}\leq \Col_R^k(x)\leq (\frac{3^{\frac{1}{2}}}{2})^kx^{1+\delta}\}$. We will show that $\frac{\lambda(A_n)}{a_{n+1}-a_n}\geq 1-\frac{C}{a_{n+1}^D}$ for some $C>0$ and $0<D<1$, then conclude that $\{x\in(0,\infty)\mid\forall k\leq\log_2x: (\frac{3^{\frac{1}{2}}}{2})^kx^{1-\epsilon}\leq \Col_R^k(x)\leq (\frac{3^{\frac{1}{2}}}{2})^kx^{1+\epsilon}\}$ is \realstardense. After that we will use an iterative argument using Lemma \ref{IterationR}, to show that $S^\lambda_\epsilon$ is \realstardense\ab for smaller and smaller $\lambda>0$.
		
		First note that	$\Col_R^k(x)\leq (\frac{3}{2})^kx\leq (\frac{3}{2})^ka_{n+1}$. Now $(\frac{3}{2})^ka_{n+1}\leq(\frac{3^{\frac{1}{2}}}{2})^ka_{n+1}^{1+\delta}$ is true as long as $(3^{\frac{1}{2}})^k\leq a_{n+1}^\delta$ or $k\leq2\delta\log_3a_{n+1}$.\\
		Also $\Col_R^k(x)\geq (\frac{1}{2})^kx\geq (\frac{1}{2})^ka_{n}$, and $(\frac{1}{2})^ka_{n}\geq (\frac{3^{\frac{1}{2}}}{2})^ka_{n}^{1-\delta}$ is true as long as $(3^{\frac{1}{2}})^k\leq a_{n}^\delta$ or $k\leq2\delta\log_3a_{n}\leq 2\delta\log_3a_{n+1}$.
		
		For $2\delta\log_3a_{n}\leq k< \log_2a_{n+1}$, and $\eta>0$ look at the set $$B_k=\{x\in(a_{n},a_{n+1})\mid \sum_{i=0}^{k-1}p(x)_i\geq (\frac{1}{2}+\eta)k\}\cup \{x\in(a_{n},a_{n+1})\mid \sum_{i=0}^{k-1}p(x)_i\leq (\frac{1}{2}-\eta)k\}.$$ Then by Lemma \ref{Hoeffding} we obtain $\mu_{(a_n,a_{n+1})}(\bigcup_{2\delta\log_3a_n\leq k\leq\log_2a_{n}}B_k)\leq \frac{C_0}{a_{n+1}^{D_0}}$ for some $C_0>0$ and $0<D_0<1$. If $x\in H=(a_n,a_{n+1})\setminus\bigcup_{2\delta\log_3a_{n}\leq k\leq\log_2a_{n}}B_k $ and $ 2\delta\log_3a_{n}\leq k<\log_2a_{n+1}$, then $\Col_R^k(x)\leq (\frac{3^{\frac{1}{2}+\eta}}{2})^k x=(\frac{3^{\frac{1}{2}}}{2})^k 3^{\eta k}x\leq (\frac{3^{\frac{1}{2}}}{2})^ka_n^\delta a_{n+1}$ if $k\eta\log_23\leq n\delta $ which is true if $\eta\log_23\leq \delta $. Similarly $\Col_R^k(x)\geq (\frac{3^{\frac{1}{2}-\eta}}{2})^k x=(\frac{3^{\frac{1}{2}}}{2})^k 3^{-\eta k}x\geq (\frac{3^{\frac{1}{2}}}{2})^ka_{n}^{1-\delta}$ if $\eta\log_23\leq \delta $. If $\delta<\epsilon$ then for large enough $n$ we obtain $a_{n}^{\delta}a_{n+1}\leq a_n^{1+\epsilon}\leq x^{1+\epsilon}$ and $a_{n}^{1-\delta}\geq a_{n+1}^{1-\epsilon}\geq x^{1-\epsilon}$. Thus we find that $\mu_{(a_n,a_{n+1})}(\{x\in(a_n,a_{n+1})\mid\forall k\leq\log_2x: (\frac{3^{\frac{1}{2}}}{2})^kx^{1-\epsilon}\leq \Col_R^k(x)\leq (\frac{3^{\frac{1}{2}}}{2})^kx^{1+\epsilon}\})\geq 1-\frac{C}{a_{n+1}^D}$ for some $C>0$ and $0<D<1$ using that $\log_2x\leq n$ for $x\in (a_n,a_{n+1})$. By lemma \ref{partitiondense} this implies that $S_\epsilon=\{x\in(0,\infty)\mid\forall k\leq\log_2x: (\frac{3^{\frac{1}{2}}}{2})^kx^{1-\epsilon}\leq \Col_R^k(x)\leq (\frac{3^{\frac{1}{2}}}{2})^kx^{1+\epsilon}\}$ is \realstardense. 
		Suppose that we already know that $S^\lambda_\epsilon=\{x\in(0,\infty)\mid\forall k\leq \frac{1-\lambda}{1-\frac{\log_23}{2}}\log_2x:  (\frac{3^{\frac{1}{2}}}{2})^kx^{1-\epsilon}\leq \Col_R^k(x)\leq (\frac{3^{\frac{1}{2}}}{2})^kx^{1+\epsilon}\}$ is \realstardense\ab for some $0<\lambda\leq1$ and all $\epsilon>0$ (this is trivially the case  for $\lambda=1$). By Lemma \ref{IterationR} we know that for any $\zeta>0$ the set $S_\epsilon^\prime=\{x\in(0,\infty)\mid \exists 0\leq l_0,l_1\leq 2\zeta \log_2x: \{3^{-l_0}\Col_R^{\lfloor\log_2x\rfloor}(x),3^{l_1}\Col_R^{\lfloor\log_2x\rfloor}(x)\}\subseteq S\}$ is \realstardense.
		Thus for $x\in S_\delta\cap S_\epsilon^\prime$ we find $0\leq l_0,l_1\leq 2\zeta \log_2x$ such that $$(\frac{3^{\frac{1}{2}}}{2})^k(3^{-l_0}\Col_R^{\lfloor\log_2x\rfloor}(x))^{1-\epsilon} \leq \Col_R^k(3^{-l_0}\Col_R^{\lfloor\log_2x\rfloor}(x))\leq (\frac{3^{\frac{1}{2}}}{2})^k(3^{-l_0}\Col_R^{\lfloor\log_2x\rfloor}(x))^{1+\epsilon} $$
		for all	$k\leq \frac{1-\lambda}{1-\frac{\log_23}{2}}\log_2(3^{-l_0}\Col_R^{\lfloor\log_2x\rfloor}(x))$
		and
		$$(\frac{3^{\frac{1}{2}}}{2})^k(3^{l_1}\Col_R^{\lfloor\log_2x\rfloor}(x))^{1-\epsilon} \leq \Col_R^k(3^{l_1}\Col_R^{\lfloor\log_2x\rfloor}(x))\leq (\frac{3^{\frac{1}{2}}}{2})^k(3^{l_1}\Col_R^{\lfloor\log_2x\rfloor}(x))^{1+\epsilon} $$
		for all $k\leq \frac{1-\lambda}{1-\frac{\log_23}{2}}\log_2(3^{l_1}\Col_R^{\lfloor\log_2x\rfloor}(x))$.\\
		Furthermore, $(\frac{3^{\frac{1}{2}}}{2})^kx^{1-\delta}\leq \Col_R^k(x)\leq (\frac{3^{\frac{1}{2}}}{2})^kx^{1+\delta}$ for all $k\leq\log_2x$. In particular,  $(\frac{3^{\frac{1}{2}}}{2})^{\lfloor\log_2x\rfloor}x^{1-\delta}\leq \Col_R^{\lfloor\log_2 x\rfloor}(x)\leq (\frac{3^{\frac{1}{2}}}{2})^{\lfloor\log_2x\rfloor}x^{1+\delta}$.
		Thus if $0\leq k\leq \frac{1-\lambda}{1-\frac{\log_23}{2}}\log_2(3^{-l_0}\Col_R^{\lfloor\log_2x\rfloor}(x))$:
		
		\begin{align*}
			(\frac{3^{\frac{1}{2}}}{2})^{k+\lfloor\log_2x\rfloor}3^{-l_0(1-\epsilon)}( (\frac{3^{\frac{1}{2}}}{2})^{\lfloor\log_2x\rfloor})^{-\epsilon} x^{(1-\delta)(1-\epsilon)}  
			&=
			(\frac{3^{\frac{1}{2}}}{2})^k(3^{-l_0}( \frac{3^{\frac{1}{2}}}{2})^{\lfloor\log_2x\rfloor}x^{1-\delta}  )^{1-\epsilon}\\
			\leq (\frac{3^{\frac{1}{2}}}{2})^k(3^{-l_0}\Col_R^{\lfloor\log_2x\rfloor}(x))^{1-\epsilon}
			&\leq
			\Col_R^k(3^{-l_0}\Col_R^{\lfloor\log_2x\rfloor}(x))\\
			\leq \Col_R^k(\Col_R^{\lfloor\log_2x\rfloor}(x)).
		\end{align*}
		
		and similarly on the other side:
		
		\begin{align*}
			(\frac{3^{\frac{1}{2}}}{2})^{k+\lfloor\log_2x\rfloor}3^{l_1(1+\epsilon)}( (\frac{3^{\frac{1}{2}}}{2})^{\lfloor\log_2x\rfloor})^{\epsilon} x^{(1+\delta)(1+\epsilon)}  
			&=
			(\frac{3^{\frac{1}{2}}}{2})^k(3^{l_1}( \frac{3^{\frac{1}{2}}}{2})^{\lfloor\log_2x\rfloor}x^{1+\delta}  )^{1+\epsilon}\\
			\geq (\frac{3^{\frac{1}{2}}}{2})^k(3^{l_1}\Col_R^{\lfloor\log_2x\rfloor}(x))^{1+\epsilon}
			&\geq
			\Col_R^k(3^{l_1}\Col_R^{\lfloor\log_2x\rfloor}(x))\\
			\geq \Col_R^k(\Col_R^{\lfloor\log_2x\rfloor}(x)).
		\end{align*}
		Given any $\epsilon^\prime>0$ we can choose small enough $\delta,\zeta,\epsilon$ with $\epsilon<\epsilon^\prime$ and $\delta<\epsilon^\prime$ such that for sufficiently large $x$: $$3^{l_1(1+\epsilon)}( (\frac{3^{\frac{1}{2}}}{2})^{\lfloor\log_2x\rfloor})^{\epsilon} x^{(1+\delta)(1+\epsilon)}\leq x^{1+\epsilon^\prime}$$ and $$3^{-l_0(1-\epsilon)}( (\frac{3^{\frac{1}{2}}}{2})^{\lfloor\log_2x\rfloor})^{-\epsilon} x^{(1-\delta)(1-\epsilon)}\geq x^{1-\epsilon^\prime}.$$
		Now, for some $K_0>0$ independent from $\lambda,\delta,\zeta,$ and $\epsilon$
		\begin{align*}
			&\lfloor\log_2 x\rfloor+\frac{1-\lambda}{1-\frac{\log_23}{2}}\log_2(3^{-l_0}\Col_R^{\lfloor\log_2x\rfloor}(x))\geq\log_2 x-1+\frac{1-\lambda}{1-\frac{\log_23}{2}}\log_2(3^{-l_0}(\frac{3^{\frac{1}{2}}}{2})^{\lfloor\log_2x\rfloor}x^{1-\delta})\\
			&	\geq \log_2x+\frac{1-\lambda}{1-\frac{\log_23}{2}}\log_2((\frac{3^{\frac{1}{2}}}{2})^{\lfloor\log_2x\rfloor}x)+\frac{1-\lambda}{1-\frac{\log_23}{2}}\log_2(3^{-l_0}x^{-\delta})-1 \\
			&\geq \log_2x+\frac{1-\lambda}{1-\frac{\log_23}{2}}(\frac{1}{2}\log_2(3)\log_2(x)-\log_2(x)+\log_2(x))+\frac{1-\lambda}{1-\frac{\log_23}{2}}\log_2(3^{-l_0}x^{-\delta})-K_0 \\
			&\geq (1+\frac{1-\lambda}{1-\frac{\log_23}{2}}\frac{1}{2}\log_2(3))\log_2x+\frac{1-\lambda}{1-\frac{\log_23}{2}}\log_2(3^{-l_0}x^{-\delta})-K_0 \\
			&\geq \frac{1-\frac{\log_23}{2}\lambda}{1-\frac{\log_23}{2}}\log_2x+\frac{1-\lambda}{1-\frac{\log_23}{2}}\log_2(3^{-l_0}x^{-\delta})-K_0 \\
		\end{align*}
		Thus, if we choose some $\frac{\log_23}{2}<q<1$ we can then make $\delta,\zeta$ small enough such that $\frac{1-\frac{\log_23}{2}\lambda}{1-\frac{\log_23}{2}}\log_2x+\frac{1-\lambda}{1-\frac{\log_23}{2}}\log_2(3^{-l_0}x^{-\delta})-K_0\geq \log_2x\frac{1-q\lambda}{1-\frac{\log_23}{2}}$ (for large enough $x$). Thus we have shown that $S^{q\lambda}_{\epsilon^{\prime}}$ is \realstardense\ab for arbitrary $\epsilon^\prime>0$. Inductively it follows that $S^{q^m}_{\epsilon^{\prime}}$ is \realstardense\ab for every $m\in\N$. Thus given any $0<\lambda<1$ we can find $m\in\N$ such that $q^m<\lambda$, thus $S^{\lambda}_\epsilon\supseteq S^{q^m}_\epsilon$ is \realstardense\ab as well. Thus the proof is complete.
	\end{proof}
	As a reformulation we note:
	\begin{thm}\label{reform}
		Suppose that $\epsilon>0$. Then the set $\{x\in(0,\infty)\mid\forall \lambda\in[0,1]:  x^{\lambda-\epsilon}\leq \Col_R^{\lfloor\frac{1-\lambda}{1-\frac{\log_23}{2}}\log_2x\rfloor}(x)\leq x^{\lambda+\epsilon}\}$ is \realstardense.
	\end{thm}
\begin{proof}
	Suppose that $\delta>0$ and set $A_\delta=\{x\in(0,\infty)\mid\forall k\leq (\frac{1}{1-\frac{\log_23}{2}})\log_2x:  (\frac{3^{\frac{1}{2}}}{2})^kx^{1-\delta}\leq \Col_R^k(x)\leq (\frac{3^{\frac{1}{2}}}{2})^kx^{1+\delta}\}$. By Theorem \ref{maintheoremdynamics} $A_\delta$ is \realstardense. Suppose that $x\in A_\delta$ and $\lambda\in[0,1]$. Then $0\leq \lfloor(\frac{1-\lambda}{1-\frac{\log_23}{2}})\log_2x\rfloor\leq (\frac{1}{1-\frac{\log_23}{2}})\log_2x$, thus $$(\frac{3^{\frac{1}{2}}}{2})^{\lfloor(\frac{1-\lambda}{1-\frac{\log_23}{2}})\log_2x\rfloor}x^{1-\delta}\leq \Col_R^{\lfloor(\frac{1-\lambda}{1-\frac{\log_23}{2}})\log_2x\rfloor}(x)\leq (\frac{3^{\frac{1}{2}}}{2})^{\lfloor(\frac{1-\lambda}{1-\frac{\log_23}{2}})\log_2x\rfloor}x^{1+\delta}.$$
	We have $(\frac{3^{\frac{1}{2}}}{2})^{\lfloor(\frac{1-\lambda}{1-\frac{\log_23}{2}})\log_2x\rfloor}x^{1+\delta}\leq\frac{2}{3^{\frac{1}{2}}} (\frac{3^{\frac{1}{2}}}{2})^{(\frac{1-\lambda}{1-\frac{\log_23}{2}})\log_2x}x^{1+\delta}=\frac{2}{3^{\frac{1}{2}}}x^{\lambda-1}x^{1+\delta}=\frac{2}{3^{\frac{1}{2}}}x^{\lambda+\delta}$,
	and on the other side $(\frac{3^{\frac{1}{2}}}{2})^{\lfloor(\frac{1-\lambda}{1-\frac{\log_23}{2}})\log_2x\rfloor}x^{1-\delta}\geq (\frac{3^{\frac{1}{2}}}{2})^{(\frac{1-\lambda}{1-\frac{\log_23}{2}})\log_2x}x^{1-\delta}=x^{\lambda-1}x^{1-\delta}=x^{\lambda-\delta}$.
	If $x$ is large enough and $\delta<\epsilon$ then we conclude that $x^{\lambda-\epsilon}\leq\frac{3^{\frac{1}{2}}}{2}x^{\lambda-\delta}\leq \Col_R^{\lfloor\frac{1-\lambda}{1-\frac{\log_23}{2}}\rfloor}(x)\leq x^{\lambda+\delta}\leq x^{\lambda+\epsilon}$, thus the claim follows since $A_\delta$ is \realstardense.
\end{proof}
As a corollary we get:
\begin{thm}
	There exists $K>0$ such that $\{x\in(\frac{3}{4},\infty)\mid (\frac{1}{1-\frac{\log_23}{2}}-\epsilon)\log_2x\leq \tau_K(x)\leq (\frac{1}{1-\frac{\log_23}{2}}+\epsilon)\log_2x\}$ has real density $1$ for all $\epsilon>0$.
\end{thm}
\begin{proof}
	We take $K$ as in Theorem \ref{Bound}. Suppose that $\delta>\eta>0$.
By Theorem \ref{reform} the set $\{x\in(0,\infty)\mid\  x^{\delta-\eta}\leq \Col_R^{\lfloor\frac{1-\delta}{1-\frac{\log_23}{2}}\log_2x\rfloor}(x)\}$ is \realstardense. Since also $\{x\in(0,\infty)\mid K<x^{\delta-\eta}\}$ is \realstardense\ab and $\delta>\eta>0$ were arbitrary, we get that $\{x\in(\frac{3}{4},\infty)\mid (\frac{1}{1-\frac{\log_23}{2}}-\epsilon)\log_2x\leq \tau_K(x)\}$ is \realstardense.  By Theorem \ref{ratetomin} the set $\{x\in(\frac{3}{4},\infty)\mid \min_{n\leq \log_2(x)(\frac{1}{1-\frac{\log_2(3)}{2}}+\epsilon)}\Col_R^n(x)\leq K\}$ has real density $1$, thus also $\{x\in(\frac{3}{4},\infty)\mid \tau_K(x)\leq (\frac{1}{1-\frac{\log_23}{2}}+\epsilon)\log_2x\}$ has real density one. As \realstardense\ab sets have real density $1$ and the intersection of two sets that have real density $1$ has again real density $1$ the proof is complete.
\end{proof}
	\section{A graph theoretic reformulation}
In this section we associate an acyclic directed graph to each $x\in\R$ and discuss some basic properties.

First assign to each real number $x\in\R$ a directed graph $H_x$ in the following way:
Let $\bold{a}^x\in\{\e,\m,0\}^\Z$ be the coefficients of the balanced ternary representation of $x$, i.e., $x=\sum_{k\in\Z}\bold{a}^x_k3^k$, (in particular $\bold{a}^x_k=0$ for all $k>N_0$ for some $N_0\in\N$). For each $n\in\Z$ consider the balanced ternary representation of $(\frac{3}{2})^nx=\sum_{k\in\Z}\bold{a}^{(\frac{3}{2})^nx}_k3^k$ (in order to ensure that the notion is well-defined, we choose representations ending in a constant $\e$-sequence when there are two representations).  The set of nodes of $H_x$ is $ \Z\times \Z$. Each node $(n,v)$ is the origin of exactly one directed edge $((n,v),(n+1,w))$, where $$w=\begin{cases}
	v&\text{if}\ab \sum_{k\geq v}\bold{a}^{(\frac{3}{2})^nx}3^k \ab\text{is \ab odd}, \\
	v+1&\text{if} \ab \sum_{k\geq v}\bold{a}^{(\frac{3}{2})^nx}3^k \ab\text{is \ab even.}
\end{cases}$$ So, intuitively, we put all the edges  $(3^zy,\Col_R(3^zy))$ into one graph for every $y$ in the orbit of $x$ under multiplication with $2$ and $3$, i.e. for all $y\in\{2^i3^jx\mid i,j\in\Z\}$.

As an example look at $x=\e\n\m\cdot\n\m\m...$. Then $H_x$ restricted to edges originating in row $0$ looks as follows:\\
$\begin{matrix}
	\hspace{16pt}\n&\e&\n&\m\cdot&\n&\m&\m&...\\	
	\hspace{16pt}\hspace{-13pt}\swarrow&\downarrow&\downarrow&\hspace{-13pt}\swarrow &\hspace{-13pt}\swarrow&\downarrow&\hspace{-13pt}\swarrow&...\\
	\hspace{16pt}\n&\e&\e&\n\cdot&\m&\e&...
\end{matrix}$\\
Now we are going to restrict $H_x$ to those vertices that have non-zero label. In order for this to be useful we note in the following lemma that if vertices with non-zero label always point to vertices who also have non-zero label:
\begin{lem}
Suppose that $x\in\R$. Set $\bold{a}^x_{n,z}:=\bold{a}^{(\frac{3}{2})^nx}_z$.	If $((n,v),(n+1,w))\in H_x$ and $\bold{a}^x_{n,z}\neq\n$ or $v=w$, then $\bold{a}^x_{n+1,w}\neq \n$. 
\end{lem}
\begin{proof}
	Suppose that $w=v$. Thus $\sum_{k\geq v}3^k\bold{a}^x_{n,k}$ is odd by definition. Let $k\geq v$ least such that  $\bold{a}^x_{n,k}\neq\n$ and $l<v$ maximal such that $\bold{a}^x_{n,l}\neq\n$. Then the value of the position $(n+1,v)$ equals $-\bold{a}^x_{n,l}\neq\n$ according to the rules of division by $2$ applied to the block $\bold{a}^x_{n,k}\n^{k-l-1}\bold{a}^x_{n,l}$ (in the case that all $\bold{a}^x_{n,l}=\n$ for $l<v$, then by our convention $\bold{a}^x_{n+1,v}=\e$).
	If $w=v+1$ and $\bold{a}^x_{n,v}\neq \n$, then the position $(n,v)$ is at the end of a block and thus $\bold{a}^x_{n+1,v+1}=-\bold{a}^x_{n,v}$.
\end{proof}
Now for every $x\in\R$ we define $G_x$ to be  $H_x$ restricted to the set $\{(n,z)\in\Z\times\Z\mid\bold{a}^x_{n,z}\neq 0\}$. By the previous lemma each non-zero vertex is the origin of exactly one directed edge.

Let $\bold{G}$ denote the set of directed graphs on $\Z\times\Z$, let $I$ be a standard Borel space and let $c:\bold{G}\rightarrow I$ be a map such that $c(G_x)=c(G_{\Col_R(x)})$ for all $x\in\R$. If the map from $(0,\infty)$ to $I$ that sends $x$ to $c(G_x)$ is Lebesgue-measurable or Baire-measurable, then it is constant on a Lebesgue-co-null set respectively a comeager set due to the ergodicity of $T$ (see Proposition \ref{T}). An example is $c_0:\bold{G}\rightarrow \N\cup\{\infty\}$ which sends $G$ to the number of $G$'s connected components. Then Theorem \ref{comeager subset} implies that $c_0$ is constantly $=1$ on a comeager set and Theorem \ref{K} implies that $c_0=k_0$ for some $k_0\in\N$ and we conjectured that $k_0=1$.

Another example of this type is the following:

For $G\in \bold{G}$ a \textit{directed branch} is a sequence $\bold{g}\in(\Z\times\Z)^\Z$ such that $ (\bold{g}_{n}\bold{g}_{n+1})\in G$ for all $n\in\Z$. Note that $G_x$ contains at least one directed branch for each $x\neq \n$, i.e., if $(i,j_i)$ is left most non-zero labeled index pair in row $i$, then either $(i,i_j)$ is a leaf, i.e., it does not occur as the second component of an edge, or it is not. If it is, then put $k_i=j_i-1$ and put $k_i=j_i$ if it is not a leaf. Then $B_x=(i,k_i)_{i\in\Z}$ is a branch in $G_x$. Now, the map $c_1$ that sends $G$ to the cardinalitiy of the set of its branches, which can attain values in $\aleph_0+1\cup\{2^{\aleph_0}\}$, is universally measurable and Baire-measurable. Thus again it is constant on a Lebesgue-co-null set and also constant on a comeager set, with possibly different values. We have:

\begin{thm}
	$G_x$ has exactly one branch for comeagerly many $x\in(0,\infty)$.
\end{thm}
\begin{proof}
	Given $n\in\N$ we show that the set $B_n$ of reals $x$ such that if $\bold{a}^x\in\{\e,\n,\m\}^\Z$ is the balanced ternary representation of $x$, then there is no branch different form $B_x$ beginning in an index left of $(0,-n)$ contains an dense open set $U_n$. The idea is that if $x=2^{-n}$ for some $n\geq 1$ then only one branch can pass through level $n$, because level $n$ has exactly one non-zero labeled index pair. Now we can approximate this situation: Given any open set $U\subseteq (0,\infty)$ we will find an open subset $V\subseteq U\cap B_n$. Assume that all reals whose balanced ternary representations begin with $\bold{a}_{n_0}...\bold{a}_0\cdot...\bold{a}_{n_1}$ are in $V$, where $n_0,n_1\in\Z$ and $\bold{a}_{n_0}\neq\n$. We know that $\{T^{-k}(1)\mid k\in\N\}$ is dense in $(\frac{3}{4},\frac{9}{4}]$. Thus there exist $k\in\N$ and $m\in\Z$ such that $3^m2^{-k}$ begins with $\bold{a}_{n_0}...\bold{a}_0\cdot...\bold{a}_{n_1}$ where $n_0,n_1\in\Z$. Then due to continuity of multiplication with $2$ there exist $l>k+N+n$ such that for any $r\in 3^m(1-2^{-l},1+2^{-l})$ the balanced ternary representation of $\frac{r}{2^k}$ begins with   $\bold{a}_{n_0}...\bold{a}_0\cdot...\bold{a}_{n_1}$, but any potential branch beginning at an index $(0,j)$ with $j>-n$ ends at level $k$ unless it is $B_x$ (where level $k$ is the set of indices $\{k\}\times\Z)$). Thus the open set $\frac{3^m}{2^k}(1-2^{-l},1+2^{-l})\subseteq V\cap B_n$. So $C=\bigcap_{n\in\N}U_n$ is as desired. 
\end{proof}
For the notion of measure we formulate the following conjecture:
\begin{prob}
	$G_x$ has exactly one branch for a Lebesgue-co-null set of $x\in(0,\infty)$.
\end{prob}
\begin{rem}
	Let $\bold{a}=(\bold{a}_{n,z})_{(n,z)\in\Z\times \Z}\in\{\e,\n,\m\}^{\Z\times\Z}$ fulfill the condition $$\bold{a}_{n+1,z+1}=-\bold{a}_{n,z}+\bold{a}_{n+1,z}+\bold{a}_{n,z-1}\bold{a}_{n+1,z}(\bold{a}_{n+1,z}+\bold{a}_{n,z-1})$$ for all $n,z\in\Z$, where addition and multiplication are modulo $3$. Define $V_{\bold{a}}=\{(n,z)\in\Z\times\Z\mid \bold{a}_{n,z}\neq\n\}$ and define the graph $G_{\bold{a}}$ with vertex set $V_{\bold{a}}$ by putting $((n,v),(n+1,w))\in G_{\bold{a}}$ for  $$w=\begin{cases}
		v+1&\text{if}\ab\ab \bold{a}_{n+1,v+1}=-\bold{a}_{n,v}, \\
		v&\text{if} \ab\ab \bold{a}_{n+1,v+1}\neq-\bold{a}_{n,v}.
	\end{cases}$$ 
	If, furthermore, for all $n\in\Z$ there exist $N_n>0$ such that $\bold{a}_{n,z}=0$ for $z>N_n$, then one can show that $G_\bold{a}$ is $G_r$ for $r=\sum_{k\in\Z}\bold{a}_{0,z}3^k$. 

Note that due to the fact that $T$ is minimal (see Proposition \ref{T}) we have the following homogeneity property: If $x,y\in\R\setminus{\{0\}}$ and $F\subset\Z\times\Z$ is finite then there exist $(a,b)\in \Z\times\Z$ such that $\bold{a}^x_{n,z}=\bold{a}^y_{n+a,z+b}$ for all $(n,z)\in F$.
\end{rem}
\begin{rem}
One can similarly define directed graphs in the case of the Collatz map on the positive integers. The Collatz conjecture is then equivalent to the Conjecture that all graphs have exactly one component. In a future work the author will show that under the hypothesis of Conjecture \ref{verylikely} for a set of natural density $1$ the corresponding graphs will have only one large component for a suitable notion of largeness.
\end{rem}
\subsection*{Acknowledgments}
The author thanks Claudius Röhl for his experimental findings regarding Conjecture \ref{verylikely} (see Remark \ref{test}). 

	\bibliography{Voll}
	\bibliographystyle{plain}
	\textit{Email address:} Manuel.Inselmann@gmx.de
\end{document}